\documentclass{amsart}
\usepackage[utf8]{inputenc}

\usepackage[utf8]{inputenc}
\usepackage{lmodern}
\usepackage{dsfont}
\usepackage{enumitem}
\usepackage[capitalize]{cleveref} %cleveref before autonum always
\usepackage{mathtools}
\usepackage{autonum}
	\setenumerate{label={\normalfont(\roman*)}, itemsep=0em} % or \upshape
       \usepackage{pgf,tikz}
\usepackage{amsfonts}
\usepackage{amsthm}
\usepackage{amsmath}
\usepackage{amssymb}
\usepackage{amscd}
\usepackage{mathrsfs}
\usepackage{todo}

\newtheorem{theorem}{Theorem}
\newtheorem{remark}[theorem]{Remark}
\newtheorem{assumption}[theorem]{Assumption}
\newtheorem{proposition}[theorem]{Proposition}
\newtheorem{lemma}[theorem]{Lemma}
\newtheorem{claim}[theorem]{Claim}
\newtheorem{corollary}[theorem]{Corollary}

\newcommand{\step}[1]{\medskip\noindent\textbf{Step #1. }}
\newcommand{\substep}[1]{\medskip\noindent\textit{Substep #1. }}

\title[Sparsity and uniform regularity for ROT]{Sparsity and uniform regularity for regularised optimal transport}

\date{November 2022}
\author[R. S. Gvalani]{Rishabh S. Gvalani$^\ast$}
\address{D-MATH, ETH Z\'Urich, R\"amistraße 101, 8092 Z\"urich, Switzerland}
\email{rgvalani@ethz.ch}
\thanks{$^\ast$Corresponding author}
\author[L. Koch]{Lukas Koch}
\address{Mathematics Department, University of Sussex, Falmer Campus, BN1 9QH Brighton, United Kingdom}
\email{lukas.koch@sussex.ac.uk}
\newcommand{\e}{\varepsilon}
\newcommand\restr[2]{{% we make the whole thing an ordinary symbol
  \left.\kern-\nulldelimiterspace % automatically resize the bar with \right
  #1 % the function
  \vphantom{|} % pretend it's a little taller at normal size
  \right|_{#2} % this is the delimiter
  }}
\newcommand{\dd}{\mathrm{d}}
\newcommand{\supp}{\mathrm{supp}}
\newcommand{\R}{\mathbb R}
\newcommand{\N}{\mathbb N}

\NewDocumentCommand{\esssup}{o}{%
  \operatorname*{ess\,sup}_{\IfValueT{#1}{#1}}%
}

\begin{document}
\begin{abstract}
We consider regularised quadratic optimal transport with subquadratic polynomial or entropic regularisation. In both cases, we prove interior  Lipschitz-estimates on a transport-like map and interior gradient Lipschitz-estimates on the potentials, under the assumption that the transport map solving the unregularised problem is bi-$C^{1,\alpha}$-regular.  For strictly subquadratic and entropic regularisation, the estimates improve to interior $C^1$ and $C^2$ estimates for the transport-like map and the potentials, respectively. Our estimates are uniform in the regularisation parameter. As a consequence of this, we obtain convergence of the transport-like map (resp. the potentials) to the unregularised transport map (resp. Kantorovich potentials) in $C^{0,1-}_{\mathrm{loc}}$ (resp. $C^{1,1-}_{\mathrm{loc}}$).

Central to our approach are sharp local bounds on the size of the support for regularised optimal transport, which we derive for a general convex, superlinear regularisation term. These bounds are of independent interest and imply global bias bounds for the regularised transport plans. Our global bounds, while not necessarily sharp, improve on the best known results in the literature for quadratic regularisation.
\end{abstract}

\maketitle

\section{Introduction} % (fold)
\label{sec:introduction}
In this paper, we study the properties of the minimisers of the following regularised optimal transport (ROT) problem:
\begin{align}\label{problem}
OT_h(\lambda,\mu)\coloneqq\min_{\pi\in \Pi(\lambda,\mu)} \int_{\R^d\times \R^d} |x-y|^2 \, \dd\pi + \e^2 \int_{\R^d \times \R^d} h\left(\frac{\dd\pi}{\dd P}\right) \,\dd P \, ,
\tag{ROT}
\end{align}
where $\lambda, \mu \in \mathcal{M}(\R^d)$ have the same mass, $\Pi(\lambda,\mu)$ is the set of all couplings between $\lambda,\mu$, $\mathcal{M}(\R^d \times \R^d) \ni P \ll \lambda \otimes \mu$ is a reference measure, and $h:[0,\infty) \to \R$ is a sufficiently nice function (see~\cref{assumptionh}), which regularises the quadratic optimal transport problem. $\e$ is the regularisation parameter and we always assume $\e\in (0,1]$. For the purposes of this paper, we fix $P=\lambda \otimes \mu$. Since we are mainly interested in the continuous setting, we will also assume throughout the paper that $\lambda,\mu$ are absolutely continuous with respect to the Lebesgue measure. We refer the reader to \cref{sec:notation} for an overview of our notation.

It is well-known that the optimal transport problem is a convex minimisation problem, but not strictly convex. The idea behind using a regularisation such as in \eqref{problem} is to obtain a strictly convex problem. Amongst other aspects, this may aid numerical computations \cite{Cuturi2013}, improve sample complexity \cite{Genevay2019,Mena2019}, and improve the stability of couplings \cite{Ghosal2022,Bayraktar2025}.

Regularised optimal transport with the choice $h(z)=z\log(z)$ is known as entropic optimal transport and has received a lot of attention in recent years. This stems from the fact that solutions to entropic optimal transport can be easily computed using robust algorithms such as Sinkhorn's algorithm \cite{Cuturi2013,Benamou2015}. Moreover, minimisers of entropic optimal transport have strong smoothness properties: in our setting, the transport plan is a smooth measure in $\R^d\times \R^d$ \cite{Nutz2022}. There is by now a vast literature on entropic optimal transport and we refer the reader to \cite{Peyre2019} for an introduction to computational aspects and \cite{Leonard2014,Nutz2022} for an introduction to the available analytical results. We will focus here on literature that is directly relevant to our results.

A draw-back of entropic regularisation is that minimisers have full support which can make it difficult to compute them efficiently. In particular, this can lead to numerical instabilities \cite{Blondel2018,Zhang2023}. This drawback of the entropic regularisation has lead to other regularisations being suggested, in particular quadratic regularisation with $h(z)=|z|^2$ \cite{Blondel2018,Montacer2018,Lingxiao2020,Zhang2023}. In this case, minimisers have numerically been observed to have sparse support. Analytical results confirming these observations have been obtained only recently in \cite{Wiesel} and in \cite{Nutz} in the scalar setting $d=1$.

This motivates us to focus on the  following choice of regularisation:
\begin{equation}
  \label{eq:lpregularisation}
  h_p(z)\coloneqq \begin{cases}\frac{|z|^p-1}{p-1} \quad& \text{ for } p\in (1,2]\\
  z \log(z) \quad& \text{ for } p=1.
  \end{cases}
\end{equation}
We remark that as $z\to 1$, $h_p(z) \to h_1(z)$. 

Some of our results hold for more general regularisations satisfying the following set of assumptions:
\begin{assumption}\label{assumptionh}
$h\colon \R\to \R$ is continuous, twice differentiable on $(0,\infty)$, strictly convex, bounded below,  and superlinear at $+\infty$.
\end{assumption}
We will see that sparsity of the support is essentially equivalent to the following additional assumption, see the discussion in \cref{rem:support} below.
\begin{assumption}\label{assumptionh2}
$h\colon \R_+\to \R$ satisfies $h'(0)=0$.
\end{assumption}

Under \cref{assumptionh}, \eqref{problem} admits a dual problem of the form
\begin{align}\label{eq:dualFormulaIntro}
\sup_{f\in L^1(\lambda),g\in L^1(\mu)} \int_{\R^d} f\dd\lambda+\int_{\R^d} g\dd\mu- \int_{\R^d\times \R^d} h^\ast\left(\frac{f(x)+g(y)-|x-y|^2}{\e^2}\right)\dd P,
\end{align}
where $h^\ast$ is the convex conjugate of $h$. It turns out that the supremum is assumed and that the minimal choice of potentials $(f_\e,g_\e)$ is unique up to constants. The connection to the minimiser $\pi_\e$ of \eqref{problem} is given through the formula
\begin{align}\label{eq:formulapif}
\frac{\dd\pi_\e}{\dd P} = (h')^{-1}\left(\frac{f_\e(x)+g_\e(y)-|x-y|^2}{\e^2}\right).
\end{align}
Our main object of interest will be the following weighted mean:
\begin{align}\label{eq:defTe}
T_\e(x)\coloneqq \left(\int h''\left(\frac{\dd\pi_\e}{\dd P}\right)^{-1}\dd\mu(y)\right)^{-1}\int y h''\left(\frac{\dd\pi_\e}{\dd P}\right)^{-1}\dd\mu(y).
\end{align}
$T_\e$ is an analogue of the transportation map, since it satisfies a Brenier-type formula:
\begin{align}\label{eq:BenamouBrenier}
T_\e(x) = x - \frac 1 2\nabla f_\e(x).
\end{align}
We refer the reader to \cref{sec:qualitative,sec:regularised_optimal_transport} for precise statements of these facts.

$T_\e$ has emerged as a natural quantity in the context of entropic regularisation \cite{Pooladian2021}, as well as for quadratic regularisation \cite{Nutz,Wiesel}.

While the smoothness of potentials in the entropic setting $h=h_1$ is well-known, available regularity results usually depend exponentially on the regularisation parameter $\e$, c.f. \cite{Nutz2022}. To our knowledge the only regularity results in the context of regularised optimal transport that are independent of $\e$ are in the setting of entropic transport. A uniform $C^{1,\alpha}$-regularity for the potentials is obtained in \cite{Baasandorj2025} for some $\alpha < 1$. Furthermore, in \cite{Gvalani2025}  a uniform large-scale Lipschitz-regularity for the optimal transportation plan (formally equivalent to large-scale gradient Lipschitz-regularity at the level of the potentials) is developed. Finally, $\e$-independent estimates on the modulus of continuity of the entropic potentials played a key role in \cite{Gozlan2025}, where they were used to deduce a generalisation of Caffarelli's contraction theorem for optimal transport.

The main result of this paper is a local Lipschitz bound on $T_\e$, independent of $\e>0$. It applies both in the setting of entropic regularisation and of subquadratic polynomial regularisation.
\begin{theorem}\label{thm:main}
Let $p\in[1,2]$ and $h(z)=h_p(z)$. Fix $\alpha\in (0,1]$. Assume $\lambda,\mu$ are compactly supported probability measures, bounded and bounded away from $0$ on their support, and with $C^{0,\alpha}$-densities. Denote $\Omega = \supp\;\lambda$. Let $T$ be the minimiser of $OT(\lambda,\mu)$ and $T_\e$ be obtained via \eqref{eq:defTe} from a minimiser of $OT_h(\lambda,\mu)$. Assume $T$ is bi-$C^{1,\alpha}$-regular on $\Omega$. Then, for any $\omega\Subset \Omega$, there is $C>0$ depending only on $d,p, \omega$, $\Omega$, $\|\lambda\|_{C^{0,\alpha}(\Omega)}$, $\|\mu\|_{C^{0,\alpha}(T(\Omega)}$ and $\|T\|_{C^{1,\alpha}(\Omega)}$ as well as $\|T^{-1}\|_{C^{1,\alpha}(T(\Omega))}$, such that
\begin{align}
\esssup[x\in \omega] |\nabla T_\e(x)|\leq C.
\end{align}
If $p<2$, $T_\e \in C^1(\omega)$ and the essential supremum may be replaced by a supremum.
\end{theorem}
We remark that due to \cite{Caf92}, if $\supp\;\lambda$ and $\supp\;\mu$ are convex, then $T$ is bi-$C^{1,\alpha}$ for some $\alpha\in (0,1]$.

As an immediate consequence of the above result and \eqref{eq:BenamouBrenier} (in the rigorous form of \cref{lem:concavity}), we have the following corollary.
\begin{corollary}\label{cor:potentials}
  Under the assumptions of~\cref{thm:main}, we have that for any $\omega\Subset \Omega$, there is $C>0$ depending only on $d,p$ and $\omega$, $\|\lambda\|_{C^{0,\alpha}(\Omega)}$, $\|\mu\|_{C^{0,\alpha}(T(\Omega))}$ and $\|T\|_{C^{1,\alpha}(\Omega)}$ as well as $\|T^{-1}\|_{C^{1,\alpha}(T(\Omega))}$ such that
  \begin{equation}
    \label{eq:potentialbound}
    \esssup[x \in \omega]|\nabla^2 f_\e| \leq C \, .
  \end{equation}
  If $p<2$, then $f_\e \in C^2(\omega)$ and the essential supremum may be replaced by a supremum.
\end{corollary}

\eqref{eq:dualFormulaIntro} suggests that the regularity of the potentials is determined by the regularity of the map $(h')^{-1}$. Note that
\begin{align}
(h'_p)^{-1}(\xi) = \begin{cases}
	\exp(\xi-1) \quad& \text{ if } p=1\\
	\frac p {p-1}\max(0,\xi)^\frac 1 {p-1}\quad & \text{ if } p>1.
\end{cases}
\end{align}
In particular, potentials are smooth when $p=1$ and their regularity decreases as $p$ increases. For $p>2$, the potentials are not expected to be twice differentiable. Hence, we expect the range of $p$ in  our theorem to be optimal. 

There are a variety of qualitative convergence results of regularised optimal transport to optimal transport. In the entropic setting, convergence of the transport plans as well as the potentials is well-established, c.f. \cite{Nutz2022}. For more general regularisations, $\Gamma$-convergence of the plans was studied in \cite{Lorenz2022}. In the case of quadratic regularisation, convergence of potentials is also known \cite{Nutz2025}. 
\cref{thm:main} immediately allows to strengthen such quantitative convergence results.
\begin{corollary}\label{cor:convergence}
  Under the assumptions of~\cref{thm:main}, we have that 
  \begin{equation}
    \label{eq:convergenceT}
    T_\e \stackrel{\e \to 0}{\to} T \, ,\quad \textrm{in $C^{0,1-}_{\rm loc}(\Omega)$} \, ,
  \end{equation}
  and 
  \begin{equation}
    \label{eq:convergencepotentials}
    f_\e \stackrel{\e \to 0}{\to} f, \, g_\e \stackrel{\e \to 0}{\to} g \, ,\quad \textrm{in $C^{1,1-}_{\rm loc}(\Omega)$} \, ,
  \end{equation}
  where $f,g$ are Kantorovich potentials of the quadratic transport problem.
\end{corollary}
We expect the above convergence results to be particularly relevant for applications.

The starting point of our proof of \cref{thm:main} is the fact that $OT_{h_p}(\lambda,\mu)$ and $OT(\lambda,\mu)\coloneqq OT_0(\lambda,\mu)$ are close in energy. Roughly speaking, 
\begin{align}\label{eq:energyCloseRough}
OT_{h_p}(\lambda,\mu)-OT(\lambda,\mu)\lesssim \e^\frac 4{d(p-1)+2}.
\end{align}
In the setting of the Schr\"odinger bridge problem, which is closely related to the entropic transportation problem, c.f. \cite{Nutz2022}, such results go back to \cite{Mikami2002,Mikami2004}. In fact, in this setting a Taylor expansion of $OT_{h_1}$ in terms of $\e$ can be obtained up to second order under strong regularity assumptions on the potentials \cite{Conforti2021}. In the entropic setting, a second-order Taylor expansion was obtained in \cite{Chizat2020}. Under weaker assumptions on the marginals, expansions to first order or estimates as stated above may still hold. We refer the reader to \cite{Malamut} for a discussion of the available results and an analysis of the separate energy contribution of the transport term and the regularisation term. In the case of polynomial regularisation the stated rates are due to \cite{Eckstein}.

Interpreting this closeness in energy through the lens of viewing minimisers of regularised optimal transport as quasi-minimisers of quadratic optimal transport, we implement a Campanato iteration to propagate this information to smaller scales. This part of our approach is based on the variational approach to optimal transportation first implemented in \cite{Goldman2020}. In \cite{Otto2017}, this approach was generalised to quasi-minimisers at all scales. However, for regularised optimal transport the quasi-minimality information is not available at all scales and in fact at small scales it is not expected to hold. In \cite{Gvalani2025} we implemented a Campanato iteration down to a critical scale for quasi-minimisers for which the energy contribution of long trajectories is controlled and obtained a large-scale $\e$-regularity result for entropic optimal transport down to scale $\e$. Here, we use the machinery of \cite{Gvalani2025} to propagate regularity down to a critical scale $R_c$ for minimisers of $h_p$-regularised transport. Finally, we analyse the structure of $T_\e$ in \eqref{eq:defTe} in order to obtain a regularity result on small scales $R<R_c$.

A scaling analysis of \eqref{problem} suggests that the regularisation becomes dominant at the critical length-scale $R_c\approx \e^\frac 2 {d(p-1)+2}$. Heuristically, minimisers of regularised optimal transport are related to the heat equation with a diffusion term of order $\e$ and Gaussians with variance $\e$, the fundamental solution of the heat equation, see e.g. \cite{Pal2024}. For polynomial regularisation, minimisers are related to the Barenblatt solutions of the porous medium equation with diffusion term $\e \nabla \rho^p$ \cite{Garriz2024}. The expansion rate of the support of the corresponding Barenblatt solutions agrees with the rate in \eqref{eq:energyCloseRough}.

In both the large-scale regime $R>R_c$ and the small-scale regime $R<R_c$ it is crucial for us to have good control over long trajectories. The following bound on the support of regularised minimisers will be crucial.

\begin{theorem}\label{prop:Linfty}
Assume $h\colon \R+\to \R$ satisfies \cref{assumptionh} and \cref{assumptionh2} and fix $L\in(0,\infty)$. Then, there exists an $R_c=R_c(\e,d,L)>0$ such that for all $R> R_c$ the following holds: 

\noindent If 
\begin{equation}
  \label{eq:locallb}
  \inf_{B_r(z)\subset B_{2R}}\frac{\lambda(B_r(z))}{r^d}\eqqcolon D >0 \, ,
\end{equation}
$\pi \in \Pi(\lambda,\mu)$ is a minimiser of \eqref{problem}, and  $b\colon B_{2R}\times \R^d$ is an $L$-Lipschitz function, then there exists a continuous function $\tau(0,\infty)\to (0,\infty)$ with $\tau(t)\to 0$ as $t\to 0$ such that 
\begin{align}
\sup_{(x,y)\in \{B_R \times \R^d \cap \supp\; \pi \}}\frac{|y-x-b(x)|}{R}\lesssim  \max\left(E(\pi,2R,b)^\frac 1 2,E(\pi,2R,b)^\frac 1 {d+2},\tau(\e)/R\right).
\end{align}
The implicit constant depends on $d,D$ and $L$.
If $h(z)=h_p(z)$ for some $p>1$, then the statement holds with $\tau(t)\sim t^\frac 2 {d(p-1)+2}$ and $R_c\sim \tau(\e)$, where the implicit constants depend on $p,d$ and $L$. 
\end{theorem}

The only available results proving sparsity of the support for regularised optimal transport problems concern $OT_{h_2}$, where global results showing that the spread of the support is of order $\tau(\e)$ are given in \cite{Nutz,Wiesel}. Our result applies to a much wider range of regularisations and is local. Whereas the proofs in \cite{Nutz,Wiesel} heavily rely on the precise structure of minimisers of quadratically regularised optimal transport, our proof builds on the local $L^\infty$-bounds for optimal transport problems obtained in \cite{Goldman2020,Goldman2025} and relies on a cyclical monotonicity result, see \cref{prop:monotonicity}. We view the terms involving $E(\pi,2R,b)$ as arising from the quasi-minimality with respect to quadratic optimal transport. They agree with the bounds obtained in the unregularised setting. The term $\tau(\e)$ comes from the regularisation and we believe it captures the sharp scaling in $\e$. Moreover, as the remark below shows, without \cref{assumptionh2} sparsity of the support fails.

\begin{remark}\label{rem:support}
The assumption $h'(0)=0$ is essentially equivalent to assuming the existence of $C_0\in \R$ such that $h'(x)\geq C_0$ for all $x\in \R_+$. Indeed, assuming the existence of such $C_0$, since $h$ is twice differentiable and convex, we may extend $h'$ continuously to $\R_+\cup \{0\}$ by setting $h'(0)=C_0$. Consider $\tilde h(x) = h(x)-C_0 x$. Then $\tilde h$ is twice differentiable, strictly convex and satisfies $\tilde h'(0)=0$. Moreover,
\begin{align}
\int \tilde h\left(\frac{\dd\pi}{\dd P}\right)\dd P =& \int h\left(\frac{\dd\pi}{\dd P}\right) \dd P-C_0\int \dd\pi=\int h\left(\frac{\dd\pi}{\dd P}\right) \dd P-C_0\lambda(\R^d).
\end{align}
In particular, minimality of \eqref{problem} is unaffected by the change $h\to \tilde h$.

If $\lim_{z\to 0} h'(z) = -\infty$ and $\pi$ is a minimiser of \eqref{problem}, then $\supp\;\pi = \supp\;P$. Indeed, it holds (see \eqref{eq:RadonNikodym}) that
\begin{align}
\frac{\dd\pi}{\dd P} = (h')^{-1}\left(\frac{f(x)+g(y)-|x-y|^2}{\e^2}\right)
\end{align}
for almost every $(x,y)\in \supp\;P$. Thus, if $\frac{\dd\pi}{\dd P}=0$ on a set of positive $P$-measure, then $f(x)+g(y)$ is infinite-valued on a set of positive $P=\lambda\otimes \mu$-measure. As $f\in L^1(\lambda), g\in L^1(\mu)$, this is impossible.
\end{remark}

While our approach does not seem to yield global bounds with the expected rate $\tau(\e)$, we can obtain a global bias bound, generalising \cite[Theorem 5.2]{Wiesel}, which treated the case $p=2$. Moreover, in the case $p=2$, we improve on the rate obtained in \cite[Theorem 5.2]{Wiesel}.
\begin{corollary}\label{cor:wieselcomparison}
Let $\lambda,\mu$ be compactly supported probability measures and assume that
\begin{equation}
\inf_{B_r(z)\subset \supp\; \lambda}\frac{\lambda(B_r(z))}{r^d} \eqqcolon D >0 \, .
\end{equation}
 Let $\pi_\e$ be a minimiser of \eqref{problem} with $h=h_p$ for some $p>1$. Assume the minimiser of the unregularised problem is given by a $L$-Lipschitz transport map $T$. Then there is $\e_0>0$ such that if $\e\leq \e_0$, it holds that
\begin{align}\label{eq:WieselComparison}
\text{either } |y-T(x)|\leq C R \quad \text{ or }\quad d(x,\partial\;\supp\;\lambda)<R \, 
\end{align}
with $R\sim \e^\frac{4}{(d(p-1)+2)(d+2)}$ and $C$ depending on $p,d,D,L$ and $\int |x|^{3}\dd(\lambda\otimes \mu)$.
\end{corollary}

The structure of this paper is as follows. In \cref{sec:prelim} we define our notation and collect a number of known results on the theory of regularised optimal transport minimisers. In \cref{sec:Linfty} we prove local bounds on the support. \cref{sec:largescale} describes how to obtain a large-scale $\e$-regularity theory, while \cref{sec:smallscale} proves gradient estimates for $T_\e$ on small-scales. Combining these two sections, we prove \cref{thm:main} in \cref{sec:proofMain}.

\section{Preliminaries}\label{sec:prelim}
\subsection{Notation}\label{sec:notation}
We write $\lesssim, \gtrsim$ if $\leq,\geq$ hold up to a  constant. We usually specify the dependence of the implicit constant in the statement of our results, but not within our proofs. We write $\sim$ if both $\lesssim$ and $\gtrsim $ hold. We use $a \ll b$ to mean that $a/b$ is sufficiently small. For $p \in [1,\infty]$, we define its H\"older conjugate $p'\coloneqq p/(p-1)$.

For $R>0$, $\#_R= (B_R\times \R^d)\times (\R^d\times B_R)$.
Given $A\subset \R^d$, we write $B\Subset A$ if $B$ is compactly contained in $A$.

 We say that a function $f:[0,\infty)\to \R$ is superlinear at $+\infty$ if
\begin{equation}
  \lim_{x\to +\infty}\frac{f(x)}{x}=+\infty \, .
\end{equation}

We use $\mathcal{M}(\R^d)$ to denote the space of positive, finite measures on $\R^d$. For any measures $\rho,\nu \in \mathcal{M}(\R^d)$, we write $\rho \ll \nu$ if $\rho$ is absolutely continuous with respect to $\nu$ and we denote by $\dfrac{\dd \rho}{\dd \nu}$ the Radon--Nikodym derivative. For any $\nu \in \mathcal{\R}^d$, we will use $\supp\; \nu$ to denote its support.  We use $\mathscr{H}^d_\mu$ to denote the $d$-dimensional Hausdorff measure defined with respect to $\mu$.

We use $C^{k,\alpha}(\Omega)$ (resp. $C^{k,\alpha}_{\rm loc}(\Omega)$) $k\in \mathbb{N},\, \alpha \in (0,1]$ to denote the space of $k$-times continuously differentiable functions on $\Omega$ (resp. compact subsets of $\Omega$) with $\alpha$-H\"older continuous highest-order derivative. We set $C^{k,\alpha-}(\Omega)=\bigcap_{\alpha'<\alpha}C^{k,\alpha'}(\Omega)$. The semi-norm in $C^{0,\alpha}(B_R)$ is denoted $[\cdot]_{\alpha,R}$.

Given a Lipschitz function $b(x)$ and $\pi\in \mathcal{M}(\R^d\times \R^d)$, denote $E(\pi,R,b) = R^{-(d+2)}\int_{B_R\times \R^d} |y-x-b(x)|^2 \dd\pi$. If $b=0$, we write $E(\pi,R)=E(\pi,R,0)$.

For $\lambda,\mu$ probability measures on $\R^d$ with $C^{0,\alpha}$-densities on $B_R$, we denote $D_{\lambda,\mu}(R)\coloneqq R^{2\alpha}([\lambda]_{\alpha,R}^2+[\mu]_{\alpha,R}^2),$ conflating measures and densities.

\subsection{A general monotonicity principle}
Monotonicity properties of minimisers to \eqref{problem} play a key role in providing $L^\infty$-bounds on their support. We state a general monotonicity principle that can essentially be found in \cite[Lemma 1.3]{Beiglbock}.

\begin{proposition}\label{prop:monotonicity}
Let $\pi$ be a minimiser of \eqref{problem}. Assume \cref{assumptionh} is satisfied.  Then for any $(x,y),(x',y')\in \supp\;\pi$ it holds that
\begin{equation}\label{eq:monotonicity}
h'\left(\frac{\dd \pi}{\dd P}(x,y)\right)+h'\left(\frac{\dd \pi}{\dd P}(x',y')\right)
\leq \frac{\Delta}{\e^2}+h'\left(\frac{\dd \pi}{\dd P}(x',y)\right)
+h'\left(\frac{\dd \pi}{\dd P}(x,y')\right)
\end{equation}
where
\begin{align}
\Delta=\Delta(x,y,x',y')\coloneqq|x-y'|^2+|x'-y|^2-|x-y|^2-|x'-y'|^2.
\end{align}
\end{proposition}
\begin{proof}
The proposition follows from an easy adaption of \cite[Lemma 1.3]{Beiglbock}.
\end{proof}

\subsection{Qualitative theory of minimisers of regularised transport}\label{sec:qualitative}
In this section, we recall the qualitative theory of minimisers of \eqref{problem}. This follows by fairly standard arguments, see e.g. \cite{Bayraktar2025,Gonzalez2025}, but for the sake of completeness we provide the complete arguments in~\cref{sec:regularised_optimal_transport}. Throughout this section, we fix a minimiser $\pi$ of \eqref{problem} and assume that \cref{assumptionh} holds. Note in particular, that then $h'$ is strictly monotonic increasing and hence has a strictly monotonic increasing inverse $(h')^{-1}$ on its image. Moreover, denoting by $h^\ast$ the convex conjugate of $h$, $(h^\ast)' = (h')^{-1}$ and for $z>0$, $(h^\ast)''(h'(z))= (h'')(z)^{-1}$.

There exist continuous potentials $f,g$, which are unique up to additive constants, such that for every $x\in \supp\;\lambda$, $y\in \supp\;\mu$, $f,g$ satisfy the following first-order optimality conditions:

\begin{align}\label{eq:Schrodinger}
1 =& \int (h')^{-1}\left(\frac{(f(x)+g(y)-|x-y|^2}{\e^2}\right)\dd\lambda(x)\\
=&  \int (h')^{-1}\left(\frac{f(x)+g(y)-|x-y|^2}{\e^2}\right)\dd\mu(y).
\end{align}
 The relation to $\pi$ is given through the identity:
\begin{align}\label{eq:RadonNikodym}
\frac{\dd\pi}{\dd P}(x,y) = (h')^{-1}\left(\frac{f(x)+g(y)-|x-y|^2}{\e^2}\right).
\end{align}
Consequently,
\begin{align}
\supp\;\pi = \{(x,y)\colon f(x)+g(y)-|x-y|^2\geq 0\}.
\end{align}
It is convenient to write $\psi(x,y) = f(x)+g(y)-|x-y|^2$. We note that then
\begin{align}
h''\left(\frac{\dd\pi}{\dd P}\right)^{-1} = (h^\ast)''(\psi(x,y)).
\end{align}

In the case $h(z)=\frac 1 {p-1}(|z|^p-1)$, more can be said about $f,g$.
\begin{lemma}\label{lem:concavity}
Let $h(z)=\frac 1 {p-1} (|z|^p-1)$ for some $p\in(1,2]$. Then, on any convex open subset $\Omega\subseteq \supp\;\lambda$, $f(x)-|x|^2$ is $C^1$ and concave.  In particular, $\psi(x,y)$ is concave as a function of $x$ on $\Omega$ for all $y$. Moreover, for $x \in \Omega$
\begin{align}
\nabla f(x) =& 2x-2\left(\int \psi(x,y)_+^{p^\prime-2}\dd\mu(y)\right)^{-1}\int y\psi(x,y)_+^{p^\prime-2}\dd\mu(y) \, .
\end{align}
An analogous result holds for $g(y)-|y|^2$ and $\psi(x,y)$ as a function of $y$.
\end{lemma}
\begin{proof}
We focus on the statements involving $f$, the arguments for the statements involving $g$ are analogous.  Set $q=p^\prime \geq 2$ and fix some convex open $\Omega \subseteq \supp\; \lambda$. Assume $f- |x|^2$ is not concave on $\Omega$. Then, there exist $x,z\in \Omega$ and $s \in (0,1)$, such that
\begin{align}
\psi(s x+(1-s) z,y)<s \psi(x,y)+(1-s) \psi(z,y) \, ,
\end{align}
for all $y \in \R^d$. We then estimate using \eqref{eq:Schrodinger}, monotonicity of $(\cdot)_+^{q-1}$, and convexity of $(\cdot)^{q-1}$,
\begin{align}
(q')^{1-q'}\e^{2(q-1)} =& \int \psi(s x+(1-s)z,y)_+^{q-1}\dd\mu(y)\\
<& \int (s \psi(x,y)_++(1-s)\psi(z,y)_+)^{q-1}\dd\mu(y)\\
\leq& s \int \psi(x,y)_+^{q-1}\dd\mu(y)+(1-s) \int \psi(z,y)_+^{q-1}\dd\mu(y) = (q')^{1-q'}\e^{2(q-1)}.
\end{align}
We conclude that $f-|x|^2$ is concave on $\Omega$ and so is $\psi(x,y)$ as a function of $x$.

Consequently, $\psi$ is differentiable in $x$ almost everywhere on $\Omega$ and for almost every $x \in \Omega$, we may differentiate \eqref{eq:Schrodinger} to find
\begin{align}
\nabla f(x) = 2x-2\left(\int\psi(x,y)_+^{q-2}\dd\mu(y)\right)^{-1}\int y \psi(x,y)_+^{q-2}\dd\mu(y).
\end{align}
Note that the above object is well-defined due to~\eqref{eq:Schrodinger}, which implies that for every $x \in \Omega$ there must exist a set of positive $\mu$-measure on which $\psi(x,y)_+>0$. Recalling that $f$ is continuous, the right-hand side is continuous as a function of $x$, proving that $f-|x|^2$ is in fact a $C^1$-function.
\end{proof}

\subsection{Transformation under affine changes of coordinates}\label{sec:affinechanges}
We outline how minimality of \eqref{problem} is affected by affine changes of coordinates. Consider $\kappa\in K$, $K$ a compact subset of $(0,\infty)$, $b\in \R^d$ and $A\in \R^{d\times d}$ positive-definite and symmetric, and $\gamma\in G$, where $G$ is a compact subset of $(0,\infty)$. Denote the set of all such $\mathsf{s}\coloneqq (A,b,\gamma,\kappa)$ by $\mathscr S$, the set of admissible rescalings. For any given $\mathsf{s}=(A,b,\gamma,\kappa)\in \mathscr S$, introduce
\begin{gather}
Q(x,y) = (Q_1(x),Q_2(y))= (A^{-1}x,\gamma a(y-b))\\
\lambda_{\mathsf s} = \kappa(Q_1)_{\#}\lambda,\quad \mu_{\mathsf s} = \kappa(Q_2)_{\#}\mu ,\quad \pi_{\mathsf s} = \kappa Q_{\#} \pi.
\end{gather}

Suppose now $\pi\in \Pi(\lambda,\mu)$ minimises \eqref{problem} with $h=h_p$ for some $p\geq 1$ and $\mathsf s=(A,b,\gamma,\kappa)\in \mathscr S$. Then
\begin{align}
&\frac 1 {\kappa} \int |x-y|^2 \dd\pi_{\mathsf s} + \frac{\e^2}{\kappa} \int h\left(\frac{\dd\pi_{\mathsf s}}{\dd(\lambda_{\mathsf s}\otimes \mu_{\mathsf s})}\right)\dd (\lambda_{\mathsf s}\otimes \mu_{\mathsf s})\\
=& \int |A^{-1} x-\gamma A(y-b)|^2 \dd\pi + \e^2 \kappa^{1-p} \int h\left(\frac{\dd\pi}{\dd P}\right)\dd P\\
&+1_{p>1}\frac 1 {p-1}\e^2(\kappa^{1-p}- \kappa)\int 1 \dd P\\
=& \int |A^{-1} x|^2-\gamma^2 |x|^2+2\gamma \langle x,b\rangle \dd \lambda + \int |\gamma A(y-b)|^2-\gamma^2 |y|^2 \dd\mu + \int \gamma |x-y|^2 \dd\pi\\
&+ \e^2\kappa^{1-p} \int h\left(\frac{\dd\pi}{\dd P}\right)\dd P+1_{p>1}\frac 1 {p-1}\e^2(\kappa^{1-p}- \kappa)\int 1 \dd P.
\end{align}
Recognising the first two integrals as null-Lagrangians, we have that $\pi_{\mathsf s}\in \Pi(\lambda_{\mathsf s},\mu_{\mathsf s})$ is a minimiser of \eqref{problem} with parameter $\e \kappa^\frac{1-p} 2\gamma^{-\frac 1 2}$.

\subsection{Quantitative convergence of costs}
We recall two results concerning convergence rates for regularised optimal transport costs to quadratic optimal transport costs.
\begin{lemma}{\cite[Corollary 3.14]{Eckstein}} \label{lem:eckstein} Let $\lambda,\mu$ be compactly supported probability measures with bounded densities $f$ and $g$, respectively. Then for $p\geq 1$,
\begin{align}
OT_{h_p}(\lambda,\mu)-OT(\lambda,\mu) \leq \kappa(\e)\coloneqq\begin{cases}
\frac d 2\e^2 \log(\e^{-2})+K\e^2 \quad& \text{ if } p=1\\
K\e^\frac 4 {d(p-1)+2} \quad& \text{ if } p>1.
\end{cases}
\end{align}
Here $K$ depends on $p$, $d$, the size of the support of $\lambda,\mu$ and the bound for their densities only. In particular, if $|\supp\;\lambda|+|\supp\;\mu|\lesssim 1$ and $|f|+|g|\lesssim 1$, then $K\lesssim 1$.
\end{lemma}
\begin{proof}
As stated in \cite[Corollary 3.14]{Eckstein}, $K$ depends on the constant in the quantization rate of $\lambda$ and $\mu$. By \cite[Corollary 6.7]{Eckstein}, see also \cite[Remark 2.1]{Eckstein}, this constant can be bounded by an explicit constant depending on $p$, $d$ and $\int |x|^{p+\delta}\dd(\lambda+\mu)$. This gives the statement.
\end{proof}
Further, we recall the following result:
\begin{lemma}\cite[Lemma 3.8]{Malamut}\label{lem:malamut}
Let $\lambda,\mu$ be compactly supported probability measures. Suppose that $T$ is the optimal transport map between $\lambda$ and $\mu$ with respect to quadratic cost and $L$-Lipschitz. Then for any $\pi\in \Pi(\lambda,\mu)$,
\begin{align}
\int |T(x)-y|^2 \dd\pi\leq 2 L \left(\int |x-y|^2 \dd \pi-\int |x-T(x)|^2\dd\lambda\right).
\end{align}
\end{lemma}

\section{Local $L^\infty$-bounds}\label{sec:Linfty}
In this section we prove our first main result, local $L^\infty$-bounds on the support of minimisers of \eqref{problem}. From this result, we deduce the bias bound \cref{cor:wieselcomparison}. Moreover, we provide a replacement for these bounds in the entropic setting that is sufficient for our needs.

\begin{proof}[Proof of \cref{prop:Linfty}]
For convenience, we write $E(R,b)\coloneqq E(\pi,R,b)$. Fix $(x,y)\in (B_R\times \R^d)\cap \supp\;\pi$. 

\step {1. Monotonicity} Recall that with $\Delta=\Delta(x,y,x',y')=|x-y'|^2+|x'-y|^2-|x-y|^2-|x'-y'|^2 $, ~\cref{prop:monotonicity} implies that for any $(x',y')\in \supp\;\pi$,
\begin{align}
0\leq & h'\left(\frac{\dd \pi}{\dd P}(x,y)\right)+h'\left(\frac{\dd \pi}{\dd P}(x',y')\right)\leq \frac{\Delta}{\e^2}+h'\left(\frac{\dd \pi}{\dd P}(x',y)\right)
+h'\left(\frac{\dd \pi}{\dd P}(x,y')\right)\\
\leq& \frac{\Delta}{\e^2}+2\max\left(h'\left(\frac{\dd \pi}{\dd P}(x',y)\right),h'\left(\frac{\dd \pi}{\dd P}(x,y')\right)\right) \, .
\end{align}
Assuming $\Delta \leq 0$, we deduce
\begin{align}\label{eq:monotonicity1}
(h')^{-1}\left(\frac{-\Delta}{2\e^2}\right)\leq& (h')^{-1}\left(\max\left(h'\left(\frac{\dd \pi}{\dd P}(x',y)\right)
,h'\left(\frac{\dd \pi}{\dd P}(x,y')\right)\right)\right)\\
=&\max\left(\frac{\dd\pi}{\dd P}(x',y),\frac{\dd\pi}{\dd P}(x,y')\right).
\end{align}

\step {2. Bounds on $\Delta$} We next prove that $-\Delta\geq r R |y-x-b(x)|$ whenever $|y-x-b(x)|\geq 18(L+1) r R + 6 \bar{r}R$ for $0<r<\frac 1 3$ and $\bar r>0$. Set
\begin{align}
B(x,y) =& \left\{(x',y')\in B_{rR}\left(x+2 r R\frac{y-x-b(x)}{|y-x-b(x)|}\right)\times \R^d\colon |x'+b(x')-y'|\leq  \bar{r} R\right\}.
\end{align}
Now fix some $(x',y')\in B(x,y)$. Note
\begin{align}
\Delta =& 2 \langle x-x',y-y'\rangle \\
=& 2 \langle x-x',y-x-b(x)-(y'-x'-b(x'))\rangle + 2\langle x-x',b(x)-b(x')+x-x'\rangle.
\end{align}
Regarding the second term, using the $L$-Lipschitz regularity of $b$,
\begin{align}
2\langle x-x',b(x)-b(x')+x-x'\rangle \leq 2(L+1)|x-x'|^2\leq 2(L+1)(3 r R)^2 \, .
\end{align}
For the first term, we find using the choice of $(x',y')$,
\begin{align}
&2\langle x-x',y-x-b(x)-(y'-x'-b(x'))\rangle \\
=& 2\langle -2 r R \frac{y-x-b(x)}{|y-x-b(x)|},y-x-b(x)-(y'-x'-b(x'))\rangle\\
&+ 2\langle x-x'+2 r R \frac{y-x-b(x)}{|y-x-b(x)|},y-x-b(x)-(y'-x'-b(x'))\rangle\\
\leq& -4 r R |y-x-b(x)|+4 r\bar{r} R^2 +2 r R(|y-x-b(x)|+\bar{r} R)\\
=& -2 r R |y-x-b(x)|+ 6 r \bar{r} R^2.
\end{align}
As $|y-x-b(x)|\geq  18(L+1) r R +6\bar{r}R$, we deduce
\begin{align}
\Delta \leq -  r R  |y-x-b(x)| \, .
\end{align}

\step{3. $B(x,y) \cap \supp\; \pi$ is large}  We want to ensure 
\begin{equation}
  \label{eq:goodSetSize}
  \pi(B(x,y))\sim (r R)^d \, .
\end{equation}

Choosing $B=B_{rR}(x+2rR\frac{y-x-b(x)}{|y-x-b(x)|})$, we have
\begin{align}\label{eq:badSet}
\pi\left(B\times \R^d\setminus B(x,y)\right)\leq& \frac 1 {(\bar{r} R)^2} \int_{B\times \R^d} |y'-x'-b(x')|^2 \dd\pi(x',y')\leq \frac{3^{d+2}R^d E(2R,b)}{\bar{r} ^2}.
\end{align}
We now divide our analysis into two regimes in which we specify $r,\bar{r}$:

\substep{1. The large transport regime} If 
\begin{equation}
|y-x-b(x)|\geq R(1 + K(d,D,L) E(2R,b)^{\frac 12}) \, ,
\end{equation}
where
\begin{equation}
  K(d,D,L)= 18 \sqrt{2D^{-1}\times 54^d}(L+1)^{-\frac d2} \, 
\end{equation}
we choose
\begin{equation}
r=1/(18(L+1))<1/3,\quad \bar{r}=K(d,D,L)/6E(2R,b)^{\frac 12}   \, .
\label{eq:regime1}
\end{equation} 
Then we deduce that
\begin{equation}
  \pi\left( B(x,y) \right) \sim (r R)^d\, ,
\end{equation}
and 
\begin{equation}
  |y-x-b(x)|\geq 18(L+1)rR + 6 \bar{r}R \, .
  \label{eq:lb1}
\end{equation}

\substep{2. The small transport regime} If  
\begin{equation}
 |y-x-b(x)|< R(1 + K(d,D,L) E(2R,b)^{\frac 12})\, , \quad E(2R,b) \leq 1\, ,
\end{equation}
we choose
\begin{align}
  r= &\,\frac{|y-x-b(x)|}{R}P(d,L),\quad P(d,L)\coloneqq \min\left(\frac{1}{4(1 + K(d,D,L))},\frac{1}{36(L+1)}\right)<\frac 13 \, , \\
  \bar{r}=&\,\sqrt{2D^{-1}} r^{-\frac d 2}3^{d/2 +1}E(2R,b)^{\frac12}\eqqcolon M(d,D,L)|y-x-b(x)|^{-\frac d2}E(2R,b)^{\frac12}   \, .
\label{eq:regime2}
\end{align} 
This implies again that
\begin{equation}
  \pi\left( B(x,y) \cap (B \times \R^d)\right) \sim ( r R)^d\, .
\end{equation}
Note further that if $|y-x-b(x)|\geq R N E(2R,b)^{\frac 1 {d+2}}$ then
\begin{align}
   6\bar{r}R \leq & \,R N^{-\frac d2 } M(d,D,L)E(2R,b)^{\frac 1 {d+2}} \\
\leq & \,  N^{-(1+\frac d2)}M(d,D,L) |y-x-b(x)| <\frac12 |y-x-b(x)|\, ,
 \end{align} 
 for $N=N(d,D,L)=(2M(d,D,L))^{1+\frac d2}$. Combining this estimate with the choice of $r$, we have
 \begin{equation}
   |y-x-b(x)|\geq 18(L+1)rR + 6\bar{r}R\, .
   \label{eq:lb2}
 \end{equation}

We now make the following claim whose proof we postpone to Step 5.
\begin{claim}\label{claim:densitybound}
For $r,\bar{r}$ chosen to be as in~\eqref{eq:regime1} or ~\eqref{eq:regime2}, there exists $C>0$, such that we can pick 
\begin{align}
(x',y')\in  \supp\; \pi \cap B(x,y)
\end{align}
such that
\begin{align}\label{eq:step4}
\max\left(\frac{\dd\pi}{\dd P}(x,y'),\frac{\dd\pi}{\dd P}(x',y)\right) \leq C (rR)^{-d}.
\end{align} 
\end{claim}

\step{4. Conclusion under~\cref{claim:densitybound}}

\substep{1. The large transport regime} 
We assume for contradiction that
\begin{equation}
  |y-x-b(x)|\geq R(1 + K(d,D,L) E(2R,b)^{\frac 12}) \, .
\end{equation}
Combining~\eqref{eq:lb1}, the outcome of Step 2, and~\eqref{eq:monotonicity1} with the choice of $(x',y')$ obtained from~\cref{claim:densitybound}, we find
\begin{equation}
  (rR)^d (h')^{-1}\left(\frac{rR |y-x-b(x)|}{\e^2}\right) \leq C \, .
\end{equation}
Using the montonicity of $(h')^{-1}$, ~\eqref{eq:lb1}, and $R\geq R_c$, we obtain
\begin{equation}
  (rR_c)^d (h')^{-1}\left(\frac{ 18(L+1)(rR_c)^2}{\e^2}\right) \leq C\, .
\end{equation}
We implicitly define $R_c$ through the expression
\begin{equation}
  (rR_c)^d (h')^{-1}\left(\frac{ 18(L+1)(rR_c)^2}{\e^2}\right) =2 C \, .
\end{equation}
This is possible because the left-hand side is a continuous, strictly monotone increasing function  which goes to $0$ as $R_c\to 0$ and $+\infty$ as $R_c \to \infty$. We thus have a contradiction, which implies that
\begin{equation}
  |y-x-b(x)|< R(1 + K(d,D,L) E(2R,b)^{\frac 12}) \, ,
\end{equation}
as long as $R\geq R_c$.  This concludes the proof if $E(2R,b) \geq 1$.

\substep{2. The small transport regime}

In light of Substep 1, we consider the setting $E(2R,b)\leq 1$ and
\begin{equation}
  |y-x-b(x)|\leq R(1 + K(d,D,L)) \, .
\end{equation}
We assume
\begin{equation}
  |y-x-b(x)|\geq RN(d,D,L)E(2R,b)^{\frac1{d+2}} \, ,
\end{equation}
where $N$ is as defined in Step 3. Combining~\eqref{eq:lb2}, the outcome of Step 2, and~\eqref{eq:monotonicity1} with the choice of $(x',y')$ from~\cref{claim:densitybound}, we obtain
\begin{equation}
  P(d,D,L)^d|y-x-b(x)|^d(h')^{-1}\left(\frac{P(d,D,L)|y-x-b(x)|^2}{\e^2}\right) \leq C \, ,
\end{equation}
for some constant $P(d,D,L)\in(0,\infty)$. Consider the continuous, strictly monotone increasing function $f:[0,\infty)\to [0,\infty)$
\begin{equation}
  f(z)=P(d,D,L)^d z^d (h')^{-1}\left(\frac{ P(d,D,L)z^2}{\e^2}\right) \, .
\end{equation}

Note that $f(0)=0$ and $\lim_{z\to \infty}f(z)=\infty$. Then, there exists a unique $\e \mapsto \tau(\e)$ such that
 $f(\tau(\e))=C$.
It is straightforward to see that $\tau(\e)\to 0$ as $\e \to 0$. We deduce that $|y-x-b(x)|\leq \tau(\e)$.  This completes the proof.

Assume finally that $h(z)=\frac 1 {p-1}(|z|^p-1)$ for some $p>1$. Then $(h')^{-1}(z) = (p^\prime|z|)^\frac 1 {p-1}$. Thus $\tau(\e)$ solves
\begin{align}
P(d,L)^d\tau(\e)^d \frac{(p^\prime P(d,L) \tau(\e)c)^\frac 1 {p-1}}{\e^\frac2 {p-1}} = C\quad\Rightarrow\quad \tau(\e)\leq c(d,L,p) \e^\frac{2}{d(p-1)+2}.
\end{align}
The same calculation shows that $R_c \sim \tau(\e)$.

\step{5. Proof of~\cref{claim:densitybound}} 
Assume for a contradiction that
\begin{align}\label{eq:contradiction}
\inf_{(x',y')\in B(x,y)\cap \supp\;\pi} \max\left(\frac{\dd\pi}{\dd P}(x,y'),\frac{\dd\pi}{\dd P}(x',y)\right)\geq C (r R)^{-d} \, ,
\end{align}
for some $C>0$ to be chosen later.
We denote 
\begin{align}
B_1(x,y) =& \{(x',y')\in B(x,y)\colon \frac{\dd\pi}{\dd P}(x',y)\geq C (r R)^{-d}\}\,,\\
B_2(x,y)=& \{(x',y')\in B(x,y)\colon \frac{\dd\pi}{\dd P}(x,y')\geq C (r R)^{-d}\}.
\end{align}
Set 
\begin{align}
\mathscr A_1 =& \{x'\in \R^d\colon \exists y' \text{ such that } (x',y')\in B_1(x,y)\cap \supp\;\pi\}\, ,\\
 \mathscr A_2 =& \{y'\in \R^d\colon \exists x' \text{ such that } (x',y')\in B_2(x,y)\cap \supp\;\pi\}.
\end{align}

Assume first that $\pi(B_1(x,y)) \geq \frac 1 2 \pi(B(x,y))$. Then, using \eqref{eq:goodSetSize}, we note
\begin{align}
(r R)^d\sim \pi(B_1(x,y))\leq \pi(\mathscr A_1\times \R^d) = \lambda (\mathscr{A}_1) \,.
\end{align}
In particular, we have
\begin{align}
1 =& \int \frac{\dd\pi}{\dd P}(x',y)\dd\lambda(x')\geq \int_{\mathscr{A}_1} \frac{\dd\pi}{\dd P}(x',y)\dd\lambda(x')\geq C (r R)^{-d} \lambda(\mathscr A_1) \gtrsim C \, .
\end{align}
For a sufficiently large choice of $C$ this provides a contradiction. A similar argument works for the case $\pi(B_2(x,y)) \geq \frac 1 2 \pi(B(x,y))$. Since $\pi(B_1(x,y)\cup B_2(x,y))=\pi(B(x,y))$, this completes the proof.

\end{proof}

We now turn to deducing the bias bound \cref{cor:wieselcomparison}.
\begin{proof}[Proof of \cref{cor:wieselcomparison}]
Using \cref{lem:malamut} and \cref{lem:eckstein}, we find
\begin{align}
\int_{\R^d\times \R^d} |y-T(x)|^2 \dd\pi_\e \leq 2L(\int_{\R^d \times \R^d}|x-y|^2 \, \dd{\pi_\e}-OT(\lambda,\mu))\lesssim L \e^\frac 4 {d(p-1)+2} \, .
\end{align}
In particular, if $B_{2R}\subset \supp\;\lambda$,
\begin{align}
E(\pi_\e,2R,T(x)-x)\lesssim LR^{-{d+2}} \e^\frac 4{d(p-1)+2} \, .
\end{align}
Applying \cref{prop:Linfty} with the choice $b(x)=T(x)-x$ to balls of size $R$ and a simple covering argument, we deduce that if $d(x,\partial\;\supp\;\lambda)\geq 2R$,
\begin{align}
\frac{|y-T(x)|}{R}\lesssim \max\left(\left(\frac{\e^\frac 4{d(p-1)+2}}{R^{d+2}}\right)^\frac 1 2, \left( \frac{\e^\frac 4{d(p-1)+2}}{R^{d+2}}\right)^\frac 1 {d+2},\frac{\e^\frac 2{d(p-1)+2}}{R}\right).
\end{align}
The choice $R= \frac 1 2\e^\frac{4}{(d(p-1)+2)(d+2)}$ concludes the proof. Note that $R>R_c$ from~\cref{prop:Linfty} as long as $\e\leq \e_0$ for some $\e_0>0$ sufficiently small.
\end{proof}

Recall that as demonstrated in \cref{rem:support} in the case of entropic optimal transport, minimisers have full support and we cannot expect a result such as \cref{prop:Linfty} to hold. An integral control on the spread of minimisers was obtained in \cite[Proposition 7]{Gvalani2025}. We state a slight refinement of the result here.
\begin{proposition}\label{prop:LinftyEntropic}
Let $\pi$ be a minimiser of \eqref{problem} with $h=h_1$. There are $C,\Lambda>0$ such that if $R\geq C \e$, $\lambda$, $\mu$ are bounded below by $c>0$ and above by $c^{-1}$ on $B_{3R}$,
 then for almost every $x\in B_R$,
\begin{align}
\frac 1 {R^2}\int_{\R^d\cap \{\lvert x-y\rvert \geq \Lambda R\}} \lvert x-y\rvert^2\frac{\dd \pi}{\dd P}(x,y)\dd \mu(y) \lesssim e^{-\frac 1 {\e^2/R^2}}(1+E(\pi_\e,2R)) \, .
\end{align}
The implicit constant depends on $d$ and $c$ only.
\end{proposition}
\begin{proof}
Note that
\begin{align}
\int_{B_R} \int_{\R^d} |x-y|^2 \frac{\dd \pi_\e}{\dd P}\dd \mu(y)\dd \lambda(x) = \int_{B_R\times \R^d} |x-y|^2 \dd \pi_\e <\infty.
\end{align}
Consequently, for almost every $x\in B_1$,
\begin{align}
\int |x-y|^2 \frac{\dd \pi_\e}{\dd P}\dd \mu(y)<\infty.
\end{align}
Fix now such $x\in B_R$. Define
\begin{align}
A(x,y)\coloneqq& \Big\{(x^\prime,y^\prime)\in \#_R\colon \lvert x-y\rvert^2+\lvert x^\prime-y^\prime\rvert^2-\lvert x^\prime-y\rvert^2-\lvert x-y^\prime\rvert^2 \geq R,\\
&\quad |x-x'|\leq 2R,\lvert x^\prime-y^\prime\rvert^2\leq  R^2 \max(\Lambda E(\pi,2R),1)\Big \}.
\end{align}
Using the monotonicity of $\pi_\e$ from \cref{prop:monotonicity}, we have
\begin{align}
&\frac{\dd \pi_\e}{\dd P}(x,y)\frac{\dd \pi_\e}{\dd P}(x^\prime,y^\prime)\leq  e^{-\frac{R^2} {\e^2} \Delta(x,y,x',y')}\frac{\dd \pi_\e}{\dd P}(x,y^\prime)\frac{\dd \pi_\e}{\dd P}(x^\prime,y).
\label{eq:approxcm}
\end{align}
In particular, using the definition of $A(x,y)$ and \eqref{eq:approxcm}, we find
\begin{equation}\label{eq:pro4step1}
\begin{split}
&\int\int_{A(x,y)}1_{|x-y|\geq \Lambda R}|x-y|^2\frac{\dd\pi_\e}{\dd P}(x,y) \dd \pi_\e(x^\prime,y^\prime)\dd \mu(y)\\
\leq& e^{-\frac{R^2}{\e^2}}\int \mathds 1_{\{\lvert x-y\rvert\geq \Lambda R\}\times A(x,y)}\lvert x-y\rvert^2\frac{\dd \pi_\e}{\dd P}(x,y^\prime)\dd \pi_\e(x^\prime,y)\dd\mu(y^\prime)\\
\lesssim& e^{-\frac {R^2}{\e^2}} \int \mathds 1_{\{\lvert x-y\rvert\geq \Lambda R\}\times A(x,y)} \lvert x-y^\prime\rvert^2 \frac{\dd \pi_\e}{\dd P}(x,y^\prime)\dd \mu(y^\prime)\dd \pi_\e(x^\prime,y)\\
&\quad + e^{-\frac {R^2}{\e^2}} \int \mathds 1_{\{\lvert x-y\rvert\geq \Lambda R\}\times A(x,y)}(\lvert x^\prime-y^\prime\rvert^2+\lvert x^\prime-y\rvert^2)\dd \mu(y^\prime)\dd \pi_\e(x^\prime,y)\\
=& I + II.
\end{split}
\end{equation}
The Schr\"odinger equation \eqref{eq:Schrodinger} allows us to deduce that
\begin{align}
II \lesssim e^{-\frac {R^2}{\e^2}}R^{d+2}(1+ E(\pi_\e,2R)).
\end{align}
Further,
\begin{align}
I\leq&  e^{-\frac{R^2} {\e^2}}(10R^2+2\Lambda R^2 E(\pi_e,2R))\int 1_{\{|x-y|\geq \Lambda\}\times A(x,y)} \frac{\dd\pi_\e}{\dd  P}(x,y') \dd\pi_\e(x',y)\dd\mu(y')\\
\lesssim& e^{-\frac {R^2} {\e^2}}R^{d+2}(1+E(\pi_e,2R)).
\end{align}
The proof of \cite[Proposition 7]{Gvalani2025} shows that $\pi(A(x,y))\gtrsim R^d$, which concludes the proof.
\end{proof}

\section{Large-scale regularity}\label{sec:largescale}
In this section we apply the large-scale regularity for almost-minimisers as developed in \cite{Gvalani2025} in order to obtain a large-scale $\e$-regularity result. Introduce
\begin{align}
\hat E(\pi,R)\coloneqq E(\pi,R) +1_{p>1}\frac{\e^2}{R^{d+2}} \int_{\#_R} h_p\left(\frac{\dd \pi}{\dd (\Pi_x\pi \otimes \Pi_y\pi) }\right) \dd (\Pi_x\pi \otimes \Pi_y\pi).
\end{align}
\begin{theorem}\label{thm:largescale} 
Fix $p \in [1,2]$ and $R_0>0$. Let $\lambda,\mu$ have $C^{0,\alpha}$-densities on $B_{R_0}$ for some $\alpha\in (0,1)$ and be bounded away from $0$ on $B_{R_0}$. Assume $\lambda(0)=\mu(0)=1$. Furthermore, let $\pi\in \Pi(\lambda,\mu)$ be a minimiser of the optimal transport problem \eqref{problem} with $h(z)= h_p(z)$. Then, there exists some $\e_1>0$ such that if
\begin{align}
\hat E(\pi,R_0)+D_{\lambda,\mu}(R_0)+\frac{\e^2}{R_0^2} <\e_1,
\label{eq:smallnessentropic}
\end{align}
then for any  $R_c \ll r\leq R_0$, we have
\begin{align}\label{eq:largeScaleReg}
\min_{A\in \R^{d \times d},b\in \R^d} \frac 1 {r^{d+2}}\int_{\#_r} \lvert y-Ax-b\rvert^2 \dd \pi\lesssim \hat E(\pi,R_0)+D_{\lambda,\mu}(R_0)+\frac{\e^\frac 4{p(d-1)+2}}{r^{2}} \, .
\end{align}
\end{theorem}

\begin{proof}
The case $p=1$ is proven in \cite[Theorem 1]{Gvalani2025}, so we focus on $p\in (1,2]$.

The result will follow from an application of \cite[Theorem 3]{Gvalani2025}. Hence, we need to check that large-trajectories are controlled and that minimisers are local almost minimisers of the quadratic energy (Assumptions (6) and (7) in \cite{Gvalani2025}): Using the notation of \cref{sec:affinechanges}, we need to prove that for any $\delta>0$, some $\alpha>0$, and any $R_c \ll R \leq R_0$:
\begin{itemize}
\item If $\pi\in \Pi(\lambda,\mu)$ is a minimiser of \eqref{problem}, for any $\mathsf s\in \mathscr S$ such that $E(\pi_{\mathsf s},2R)\ll 1$ and any $\hat \pi = \tilde \pi + \restr{\pi_{\mathsf s}}{\#_R^c}\in \Pi(\lambda_{\mathsf s},\mu_{\mathsf s})$,
\begin{align}\label{eq:quasimin}
&\int_{\#_R} |x-y|^2 \dd\pi_{\mathsf s}+\e^2 \int_{\#_R} h\left(\frac{\dd\pi}{\dd P_{\mathsf s}}\right)\dd P_{\mathsf s}-\int |x-y|^2 \dd\tilde \pi 
\leq C \e^\frac 4{d(p-1)+2} \pi_{\mathsf s}(\#_R)\\ &\,\qquad\\ &\,\qquad +\delta \left(\int_{\#_{2R}} |x-y|^2 \dd\pi_{\mathsf s}+\e^2 \int_{\#_{2R}}h\left(\frac{\dd \pi_{\mathsf s}}{\dd P_{\mathsf s}}\right)\dd P_{\mathsf s}\right).
\end{align}
Here $P_{\mathsf s} = \lambda_{\mathsf s}\otimes \mu_{\mathsf s}$.
\item If $\pi\in \Pi(\lambda,\mu)$ is a minimiser of \eqref{problem}, then there exist $\Lambda>0$ such that for all $\mathsf s\in \mathscr S$, if $E(\pi_{\mathsf s},2R)+D_{\lambda_{\mathsf s},\mu_{\mathsf s}}(2R)\ll 1$, then
\begin{align}
\frac 1 {R^{d+2}} \int_{\#_R\cap \{|x-y|\geq \Lambda R\}}|x-y|^2 \dd\pi_{\mathsf s}\leq C \delta E(\pi_{\mathsf s},2R).
\end{align}
\end{itemize}
In fact, the first condition appears in \cite{Gvalani2025} without the terms involving $h$. Denote the marginals of $\pi\in \mathscr M(\R^d\times \R^d)$ by $\Pi_x \pi$ and $\Pi_y \pi$, respectively. Arguing exactly as in the proof of \cite[Theorem 6]{Gvalani2025}, \eqref{eq:quasimin} allows to establish a one-step improvement of the form: For any $\beta\in (0,1)$, there exists $\theta\in (0,1)$ and $\mathsf s_1\in \mathcal S$ such that
\begin{align}
\hat E(\pi_{\mathsf s_1},\theta R)\leq \theta^{2\beta}\hat E(\pi_{\mathsf s},10R)+C_\theta  D_{\lambda_{\mathsf s},\mu_{\mathsf s}}(10R)+C_\theta R^{-2}\e^\frac 4 {d(p-1)+2}.
\end{align}
The Campanato iteration employed in \cite[Theorem 3]{Gvalani2025} now gives \eqref{eq:largeScaleReg} without any changes to the argument, except to replace $E$ with $\hat E$.

In light of the discussion of \cref{sec:affinechanges}, it suffices to prove the claims for the case $\mathsf s = (\textup{Id},0,1,1)$. Hence, we consider only this case and drop the subscript. The second item is then a direct consequence of \cref{prop:Linfty}.

Hence, it remains to prove the local quasi-minimality of minimisers with respect to the quadratic cost \eqref{eq:quasimin}. Define the local marginals
\begin{align}
\overline \lambda(A) = \restr{\pi}{\#_R}(A\times \R^d),\quad \overline \mu(A) = \restr{\pi}{\#_R}(\R^d\times A).
\end{align}
Note that $\overline \lambda \leq \lambda, \overline \mu\leq \mu$. Let $\overline \pi= \textup{argmin}\; OT_h(\overline \lambda,\overline \mu)$. We remind the reader that in our notation $OT_h(\overline \lambda,\overline \mu)$ is considered with respect to the reference measure $\overline P = \overline \lambda\otimes \overline \mu$.  Denote $\mathscr A=\supp\; \restr{\pi}{\#_R^c} \cap \supp\;\overline \pi$. 

Consider $\hat \pi =\restr{\pi}{\#_R^c}+\overline \pi$. Note that $\Pi_x \hat \pi  = \lambda$, $\Pi_y \hat \pi  = \mu$. Hence,
\begin{align}
&\int |x-y|^2 \dd\hat\pi + \e^2 \int h\left(\frac{\dd\hat\pi}{\dd P}\right)\dd P\\
=& \int_{\#_R^c} |x-y|^2 \dd\pi + \e^2 \int_{\#_R^c\setminus \mathscr A} h\left(\frac{\dd\pi}{\dd P}\right)\dd P + \int |x-y|^2 \dd \overline \pi \\
&\quad+ \e^2 \int_{(\R^d\times \R^d)\setminus\mathscr A} h\left(\frac{\dd\overline \pi}{\dd P}\right)\dd P + \e^2\int_{\mathscr A} h\left(\frac{\dd(\restr{\pi}{\#_R^c} + \overline \pi)}{\dd P}\right) \, \dd P \\
\leq& \int_{\#_R^c} |x-y|^2 \dd\pi + \e^2 \int_{\#_R^c\setminus \mathscr A} h\left(\frac{\dd\pi}{\dd P}\right)\dd P + OT_h(\overline \lambda,\overline \mu) \\
&\quad + \e^2\int_{\mathscr A} h\left(\frac{\dd(\restr{\pi}{\#_R^c} + \overline \pi)}{\dd P}\right)\dd P + C\e^2 R^d \,.\label{eq:localMinimalityMain}
\end{align}

We now estimate the last term as follows: Note that due to \cref{prop:Linfty}, $\mathscr A\subset (B_{R+r}\times B_{R+r})\setminus \left(B_{R-r}\times B_{R-r}\right)$ for $r=C R \max\left(E(\pi,2R)^\frac 1 {d+2},R^{-1} \e^\frac 2 {d(p-1)+2}\right)$. Consequently, for any $\delta>0$,
\begin{align}\label{eq:lastTerm}
&\int_{\mathscr A} h\left(\frac{\dd(\restr{\pi}{\#_R^c} + \overline \pi)}{\dd P}\right)\dd P\leq \frac 1 {p-1}\int_{\mathscr A} \left| \frac{\dd(\restr{\pi}{\#_R^c} + \overline \pi)}{\dd P}\right|^p\dd P+ C \e^2 R^d\\
\leq & \, \frac{1+\delta}{p-1} \int_{\mathscr A} \left| \frac{\dd \restr{\pi}{\#_R^c}}{\dd P}\right|^p\dd P + C(\delta) \int_{\mathscr A} \left| \frac{\dd\overline \pi}{\dd\overline P}\right|^p\dd \overline P +C R^d\e^2\,.
\end{align}
In order to apply \cref{lem:eckstein}, we rescale $\bar \pi$. Consider
\begin{align}
\tilde \pi = \pi(\#_R)^{-1}(x/R,y/R)_{\#} \bar \pi \in \Pi(\tilde \lambda,\tilde \mu)\coloneqq \Pi\left(\pi(\#_R)^{-1} (x/R)_{\# }\bar \lambda,\pi(\#_R)^{-1} (y/R)_{\# }\bar \mu\right).
\end{align}
 Note that $\tilde \lambda$ and $\tilde \mu$ are probability measures, supported in $B_2$ and have bounded densities. We compute
\begin{align}
&\int |x-y|^2 d \overline \pi + \frac{\e^2}{p-1} \int \left(\left\lvert \frac{\dd\overline \pi}{\dd \overline P} \right \rvert^p-1\right) \dd \overline P\\
=& \pi(\#_R) R^2 \int |x-y|^2 \dd\tilde \pi+\frac{\e^2 \pi(\#_R)^{2-p}}{p-1} \int \left(\left\lvert \frac{\dd\tilde \pi}{\dd \tilde P} \right \rvert^p-1\right) \dd \tilde P\\
&\quad+\frac{\e^2}{p-1}(\pi(\#_R)^{2-p}-\pi(\#_R)^2)
\end{align}
In particular $\tilde \pi$ minimises $OT_h(\tilde \lambda,\tilde \mu)$ with regularisation parameter 
\begin{align}
\tilde \e^2 = \e^2 \pi(\#_R)^{1-p}R^{-2}\sim \e^2 R^{d(1-p)-2} \ll 1
\end{align}
 as $R \gg R_c$. Applying \cref{lem:eckstein}, we deduce that there is $K>0$ depending only on $p$, $d$ and the bounds on the densities of $\lambda,\mu$ such that
\begin{align}\label{eq:quasiminimalityApplied}
OT_h(\overline \lambda,\overline \mu) \leq \pi(\#_R)R^2 \left(OT(\tilde \lambda,\tilde \mu)+K\tilde \e^\frac 4 {d(p-1)+2}\right)\leq OT(\overline \lambda,\overline \mu)+K C \e^\frac 4 {d(p-1)+2} R^d.
\end{align} 

Combining \eqref{eq:localMinimalityMain} with \eqref{eq:lastTerm}, re-arranging and using \eqref{eq:quasiminimalityApplied}, we find
\begin{align}
&\int_{\#_R} |x-y|^2 \dd \pi+\e^2 \int_{\#_R} h\left(\frac{\dd \pi}{\dd P}\right) \dd P \\
\leq& \delta\left(\int_{\#_{2R}} |x-y|^2 \dd\pi+\e^2 \int_{\#_{2R}} h\left(\frac{\dd \pi}{\dd P}\right)\dd P\right) + OT_h(\overline \lambda,\overline \mu) + C(\delta)R^d\e^2 \\
&\, + C(\delta) \e^2 \int_{\mathscr A} h\left( \frac{\dd\overline \pi}{\dd\overline P}\right)\dd \overline P \\ 
\leq& \delta\left(\int_{\#_{2R}} |x-y|^2 \dd\pi+\e^2 \int_{\#_{2R}} h\left(\frac{\dd \pi}{\dd P}\right)\dd P\right) + OT(\overline \lambda,\overline \mu) + C(\delta) R^d\e^\frac 4 {d(p-1)+2}.
\end{align}
This establishes \eqref{eq:quasimin} and completes the proof.
\end{proof}

\section{Small-scale regularity}\label{sec:smallscale}
The goal of this theory is to provide a Lipschitz estimate on $T_\e$ at the critical scale $\e^\frac 2 {d(p-1)+2}$. The proof relies on the qualitative theory of minimisers to \eqref{problem} explained in \cref{sec:qualitative}.

\begin{proposition}\label{prop:smallScale}
Let $p\in [1,2]$. Suppose $\pi$ is a minimiser of \eqref{problem} with $h(z)=h_p(z)$. Then there is $\e_1>0$ and $C=C(p,d)>0$ such that if $\e \leq \e_1$ and $r= C  \e^\frac {2}{d(p-1)+2}$, $\lambda,\mu$ are bounded above and below on $B_r$ and $E(\pi,B_r)\ll 1$, then
$T_\e$ is Lipschitz in $B_{r/3}$ with a Lipschitz bound independent of $\e$:
\begin{align}
\esssup[x\in B_{r/3}] |\nabla T_\e]\lesssim 1 \, .
\end{align}
If $p<2$, $T_\e\in C^2(B_{r/3})$ and the essential supremum may be replaced by a supremum.
\end{proposition}

\begin{proof}
Set $S(x) = \{y\colon (x,y)\in \supp\;\pi\}$ and $\hat S(y) = \{x\colon (x,y)\in \supp\;\pi\}$. Let $\delta\in (0,1/6)$. For $p>1$, choosing $C>0$ sufficiently large, \cref{prop:Linfty} ensures that
% Note that
%  $r= C \e^\frac 2 {d(p-1)+2} = C \e^\frac{2(q-1)}{d+2(q-1)}$ where $q=p'$.
 
\begin{align}
\sup_{x\in B_{r/3}, y\in S(x), \hat x\in \hat S(y)} |\hat x-y|\leq \delta r  \eqqcolon L_\infty.
\end{align}
We now divide our analysis into three cases, $p\in(1, 2)$, $p=2$ and $p=1$. We set $q=p'$ for the rest of the proof. We note that in terms of $q$, $r=C \e^\frac{2(q-1)}{d+2(q-1)}$. We also now fix some $x \in B_{r/3}$.

\step{1. $p\neq 2$}  Note that we may use \cref{lem:concavity} to rewrite $T_\e$ as
\begin{align}\label{eq:firstDerivative}
T_\e(x)-2x = 2\left(\int \psi(x,y)_+^{q-2}\dd\mu(y)\right)^{-1}\int (y-x) \psi(x,y)_+^{q-2}\dd\mu(y)= \nabla f_\e(x).
\end{align}
We deduce,
\begin{align}\label{eq:boundFirstDerivative}
|\nabla f_\e(x)| \leq 2 L_\infty \, .
\end{align}
\cref{lem:concavity} shows that $\psi$ is a concave function of $x$. This justifies differentiating \eqref{eq:firstDerivative} to obtain for almost every $x\in B_{r/3}$,
\begin{align}
&\frac 1 2\nabla T_\e(x)\\
 =& (q-2)\left(\int_{S(x)}\psi(x,y)^{q-2} \dd\mu(y)\right)^{-1} \int_{S(x)} (y-x)\otimes\nabla_x \psi(x,y) \psi(x,y)^{q-3} \dd\mu(y)\\
&+(q-2)\left(\int_{S(x)}\psi(x,y)^{q-2} \dd\mu(y)\right)^{-2}\int_{S(x)} \nabla_x \psi(x,y)\psi(x,y)^{q-3}\dd\mu(y)\\
&\quad\otimes \int_{S(x)} (y-x) \psi(x,y)^{q-2} \dd\mu(y)= I + II.
\end{align}
We fix now $x\in B_{r/3}$ for which the above identity holds.
We begin by noting that
\begin{align}\label{eq:boundDerivpsi}
\sup_{y \in S(x)}|\nabla_x \psi(x,y)|= |\nabla f_\e(x) - 2(x-y)|\leq |\nabla f_\e(x)|+2\sup_{y\in S(x)}|x-y|\leq 4 L_\infty \, ,
\end{align}
where we used \eqref{eq:boundFirstDerivative}.

Let $\psi_m $ be the maximum of $\psi(x,\cdot)$. As $\psi(x,\cdot)$ is concave on $S(x)$ and not constant, this is well-defined and moreover assumed at a point $y_m \in S(x)$. By concavity again, we find
\begin{align}
\psi_m = \psi(x,y_m)\leq \psi(x,y)+\langle \nabla_y \psi(x,y),y-y_m\rangle
\end{align}
for any $y\in S(x)$. We deduce, using a version of  \eqref{eq:boundFirstDerivative} for $g_\e$, that there is $c>0$ such that
\begin{align}
\label{eq:concaveUsed}
\psi(x,y)\geq \psi_m - c L_\infty |y-y_m|.
\end{align}

Consequently, as $q\geq 2$, we may estimate
\begin{align}\label{eq:lowerBoundSmallScale}
\int \psi(x,y)_+^{q-2} \dd \mu(y)\gtrsim& \int (\psi_m - c L_\infty |y-y_m|)_+^{q-2} \dd y\\
\gtrsim& \int (\psi_m - c L_\infty t)_+^{q-2} t^{d-1} \dd t\gtrsim \frac{\psi_m^{q-2+d}}{L_\infty^d}.
\end{align}
Further note for $k\leq q$, using \cref{prop:Linfty},
\begin{align}
\int_{S(x)} \psi(x,y)^{q-k} \dd \mu(y)\lesssim \psi_m^{q-k} |S(x)|\lesssim \psi_m^{q-k} L_\infty^d.
\end{align}
Finally, if $q\in (2,3)$, we estimate using \eqref{eq:concaveUsed},
\begin{align}
\int_{S(x)} \psi(x,y)^{q-3} \dd\mu(y)\lesssim \int (\psi_m - c L_\infty |y-y_m|)_+^{q-3} \dd y\lesssim& \frac{\psi_m^{q-3+d}}{L_\infty^d}.
\end{align}
Collecting estimates for $q\geq 3$, we have shown
\begin{align}\label{eq:Estimate1}
I + II \lesssim \frac{L_\infty^2 \psi_m^{q-3}L_\infty^d}{\psi_m^{q-2+d} L_\infty^{-d}}= \frac{L_\infty^{2d+2}}{\psi_m^{d+1}}.
\end{align}
Recall that due to \eqref{eq:Schrodinger},
\begin{align}\label{eq:SchrodingerUsed}
\e^{2(q-1)} \lesssim \int \psi(x,y)_+^{q-1}\dd\mu(y)\leq \psi_m^{q-1} |S(x)|\lesssim \psi_m^{q-1} L_\infty^d.
\end{align}
Combining \eqref{eq:Estimate1} and \eqref{eq:SchrodingerUsed} gives
\begin{align}\label{eq:finalEstimateSmallScale}
I+ II \lesssim \frac{L_\infty^{2d+2}}{\e^{2(d+1)}L_\infty^{-\frac{d(d+1)} {q-1}}} = \frac{L_\infty^{2d+2+\frac{d(d+1)} {q-1}}}{\e^{2(d+1)}}\lesssim 1,
\end{align}
by definition of $L_\infty$ and the choice of $r$. In case $q<3$, collecting estimates and using \eqref{eq:SchrodingerUsed}, we find
\begin{align}
I + II \lesssim \frac{L_\infty^2 \psi_m^{q-3+d}L_\infty^{-d}}{\psi_m^{q-2+d} L_\infty^{-d}}= \frac{L_\infty^2}{\psi_m}\lesssim \frac{L_\infty^{2+\frac d {q-1}}}{\e^2}\lesssim 1.
\end{align}

It remains to prove that $T_\e$ is indeed everywhere differentiable. In order to see this, it suffices to prove that $I+II$ is continuous as a function of $x$. For $q\geq 3$, this is obvious. On the other hand, for $q\leq 3$, following the arguments leading up to \eqref{eq:lowerBoundSmallScale}, we find for $a>0$,
\begin{align}
\int_{S(x)\cap \{\psi(x,y)\leq a\}} \psi(x,y)^{q-3}\dd \mu(y)\lesssim \frac{\psi_m(x)^{q-3+d}}{L_\infty^d}\int_0^{a} (1-t)^{q-3}t^{d-1}\dd t.
\end{align}
In particular, since $\psi$ is a $C^1$-function of $x$,
\begin{align}
&\sup_{x\in \overline B_{r/3}} \int_{S(x)\cap \{\psi(x,y)\leq a\}} \psi(x,y)^{q-3}\dd \mu(y)\\
\lesssim& \frac{(\sup_{x\in \overline B_{r/3}}\psi_m(x))^{q-3+d}}{L_\infty^d}\int_0^{a} (1-t)^{q-3}t^{d-1}\dd t<\infty.
\end{align}
We deduce that $\int_{S(x)} \psi(x,y)^{q-3}\dd\mu(y)$ is uniformly integrable in $x$ on $\overline B_{r/3}$ from which, in combination with the bounds obtained above, it follows that $I+II$ is continuous on $B_{r/3}$ by Vitali's convergence theorem.

\step{2. $p=2$} The case $p=2$ is slightly simpler. We find arguing as before, that for almost every $x\in B_{r/3}$,
\begin{align}
\frac 12\nabla T_\e(x) =& |S(x)|^{-1}\int_{\partial S(x)} \frac{\partial_x \psi}{|\partial_x \psi|}\otimes (y-x) \dd \mathscr H^{d-1}_\mu(y)\\
&-|S(x)|^{-2} \int_{\partial S(x)} \frac{\partial_x \psi}{|\partial_x \psi|} \dd \mathscr H^{d-1}_\mu(y)\otimes \int_{S(x)}(y-x)\dd \mu(y)= I + II.
\end{align}
Fix an $x$ for which the above identity holds. In particular,
\begin{align}
I + II \lesssim L_\infty \frac{|\partial S(x)|}{|S(x)|}\lesssim \frac{L_\infty^d}{|S(x)|}.
\end{align}
To obtain the last inequality, we used that $S(x)$ is a convex set contained in $B_{c L_\infty}(x)$ for some $c>0$ so that $|\partial S(x)|\leq |\partial B_{c L_\infty}(x)|$. Note that due to the analogous estimate for $g_\e$ to \eqref{eq:boundFirstDerivative} , we have
\begin{align}\label{eq:psim1}
|S(x)|\gtrsim \left(\frac{\psi_m}{L_\infty}\right)^d.
\end{align}
On the other hand, \eqref{eq:SchrodingerUsed} gives
\begin{align}\label{eq:psim2}
\e^2 \leq \psi_m |S(x)| \quad \Leftrightarrow \quad |S(x)| \geq \frac{\e^2}{\psi_m}.
\end{align}
Optimising in $\psi_m$ over \eqref{eq:psim1} and \eqref{eq:psim2}, we deduce $|S(x)|\gtrsim L_\infty^{-\frac d {d+1}} \e^\frac{2d}{d+1}$. Thus, we have, due to our choice of $r$ and the definition of $L_\infty$,
\begin{align}
I + II\lesssim \frac{L_\infty^{d+\frac d {d+1}}}{\e^\frac{2d}{d+1}}\lesssim 1 \, .
\end{align}

\step{3. $p=1$} In the case $p=1$, \eqref{eq:Schrodinger} take the form
\begin{align}
f_\e(x) = - \e^2\log \int \exp((g_\e(y)-|x-y|^2)/\e^2)\dd\mu(y) \quad \text{ for } x\in \supp\;\lambda\, ,\\
g_\e(y) = - \e^2\log \int \exp((f_\e(x)-|x-y|^2)/\e^2)\dd\lambda(x)\quad \text{ for } y\in \supp\;\mu.
\end{align}
and it holds
\begin{align}
\frac{\dd\pi_\e}{\dd P} = \exp((f_\e(x)+g_\e(y)-|x-y|^2)/\e^2)
\end{align}
for every $(x,y)\in \supp\;P$.
It is an immediate consequence that $f_\e$ and $g_\e$ are smooth on $\supp\;\lambda$ and $\supp\;\mu$ respectively. We compute for $x\in B_{r/3}$,
\begin{align}
\nabla f_\e (x) = \int (x-y) \frac{\dd \pi_\e}{\dd P}\dd\mu(y). 
\end{align}
Using H\"older's inequality, \cref{prop:LinftyEntropic} and the choice of $r$, it follows that for almost every ${x\in B_{r/3}}$,
\begin{align}\label{eq:firstDerivativeEntropic}
|\nabla f_\e(x)|\leq \left(\int_{|x-y|\geq \Lambda r} |x-y|^2 \frac{\dd\pi_\e}{\dd P} \dd\mu(y)\right)^\frac 1 2+ \left(\int_{|x-y|\leq \Lambda r} |x-y|^2 \frac{\dd\pi_\e}{\dd P} \dd\mu(y)\right)^\frac 1 2\lesssim \e.
\end{align}
Further,
\begin{align}
\nabla^2 f_\e(x) = \e^{-2}\nabla f_\e(x)\otimes \nabla f_\e(x)+\e^{-2} \int (x-y)\otimes (x-y) \frac{\dd\pi}{\dd P}\dd P
\end{align}
and we deduce, using \eqref{eq:firstDerivativeEntropic}, $|\nabla^2 f_\e(x)|\lesssim 1$ for almost every $x\in B_{r/3}$. As $f_\e$ is $C^2$, this completes the proof.
\end{proof}

By essentially the same argument, we obtain regularity in the regime of large regularisation parameter.
\begin{lemma}\label{lem:RegularitylargeEps}
Let $p\in [1,2]$, $h(z)=h_p(z)$ and $\e \geq \e_1>0$. Assume $\lambda,\mu$ are probability measures, bounded and bounded away from $0$ on their support. Denote $\Omega = \supp\;\lambda$. Let $T_\e$ be obtained via \eqref{eq:defTe} from a minimiser $\pi_\e$ of $OT_h(\lambda,\mu)$. Then there is $C>0$ depending only on $d,p, \e_1$ and the size of the support of $\lambda,\mu$ such that
\begin{align}
\esssup[x\in \Omega] |\nabla T_\e(x)|\leq C.
\end{align}
Moreover, if $p<2$, $T_\e \in C^1(\Omega)$ and the essential supremum may be replaced by a supremum.
\end{lemma}
\begin{proof}
The proof follows from calculations similar to those in the proof of \cref{prop:smallScale}.  Hence, we only present the case $p\in (1,3/2)$. Denoting $q=p^\prime$, this means $q\geq 3$. The other cases follow analogously.

By \cref{lem:LipschitzRegularityPotential}, there is $C_1>0$ depending only on the size of the support of $\lambda,\mu$ such that
\begin{align}
\esssup[x\in \Omega]|\nabla f_\e|+\esssup[y\in \supp\;\mu]|\nabla g_\e| \leq C_1.
\end{align}
Further set $L_\infty =\sup_{(x,y)\in \supp\;\pi_\e}|x-y|$. Arguing as in the proof of \cref{prop:smallScale} up to \eqref{eq:finalEstimateSmallScale}, we find for almost every $x\in \Omega$,
\begin{align}
|\nabla T_\e(x)|\lesssim (L_\infty+C_1) \frac{L_\infty^{2d+1}}{\e^{2(d+1)}L_\infty^{-\frac{d(d+1)}{q-1}}}\leq \frac{(L_\infty+C_1)^{2d+2+\frac{d(d+1)}{q-1}}}{\e_1^{2(d+1)}}.
\end{align}
Noting that both $L_\infty$ and $C_1$ may be bounded by a bound only depending on the size of the support of $\lambda$ and $\mu$ this concludes the proof. The moreover statement may be obtained arguing as in the proof of \cref{prop:smallScale}.
\end{proof}

\section{Proof of~\cref{thm:main} and \cref{cor:convergence}}
\label{sec:proofMain}
We now combine the large and small-scale regularity results from~\cref{sec:largescale,sec:smallscale} respectively, to derive a global regularity estimate for $T_\e$.
\begin{proof}[Proof of \cref{thm:main}]
For any $\e_1>0$, there exists $R_0>0$ such that for $R\leq R_0$, $R^{2\alpha}([\lambda]_{\alpha,R}^2+[\mu]_{\alpha,R}^2)< \e_1$, whenever $B_R(x)\subset \Omega$. Moreover, due to the $C^{1,\alpha}$-regularity of $T$, we may find $A$ invertible and $b$ with $|(I+A)^{-1}|+|(I+A)|+|b|\lesssim \|T^{-1}\|_{C^{1,\alpha}(\Omega)}+\|T\|_{C^{1,\alpha}(\Omega)}$ such that for $B_R\subset\Omega$,
\begin{align}\label{eq:approximateInitial}
\int_{B_R} |T(x)-x-Ax-b|^2\dd\lambda(x)\lesssim R^{d+2+2\alpha}.
\end{align}
Consequently,
\begin{align}
\int_{\#_R} |y-x-Ax-b|^2 \dd\pi_\e\lesssim &\int_{\#_R} |y-T(x)|^2 \dd\pi_\e+\int_{\#_R} |T(x)-x-Ax-b|^2\dd\pi_\e(x)\\
=&\int_{\#_R} |y-T(x)|^2 \dd\pi_\e + \int_{B_R} |T(x)-x-Ax-b|^2 \dd\lambda\\
\lesssim& \int_{\#_R} |y-T(x)|^2 \dd\pi_\e + R^{d+2+2\alpha}.
\end{align}
As $T$ is $L$-Lipschitz on $\Omega$ for some $L>0$, combining \cref{lem:malamut} and \cref{lem:eckstein} we find
\begin{align}
\int |y-T(x)|^2 \dd\pi_\e \leq& L \left(\int |y-x|^2 \dd\pi_\e-\int |T(x)-x|^2 \dd\lambda\right)\lesssim \kappa(\e),
\end{align}
where $\kappa(\e)\to 0$ as $\e\to 0$.

Reducing first $R_0$ if necessary, then reducing $\e_1$ if necessary,  we may assume that for any ball  $B_{R_0}(z)$ contained in $\Omega$ there exists some $A$ invertible and $b$ with $|A^{-1}|+|A|+|b|\lesssim \|T^{-1}\|_{C^{1,\alpha}(\Omega)}+\|T\|_{C^{1,\alpha}(\Omega)}$, such that
\begin{align}
\int_{(B_{R_0}(z)\times \R^d )\cup( \R^d \times B_{R_0}(z) )} |y-x-Ax-b|^2 \dd\pi_\e+D_{\lambda,\mu}(z,R_0)\leq \e_1.
\end{align}
Fix any such ball and change coordinates so that $z=0$. Introduce $\tilde \pi_\e = \pi_{\mathsf s}$ with $\mathsf s = (A+\textup{id},b,\left(\frac{\mu(0)}{\lambda(0)}\right)^\frac 1 d,\lambda(0)^{-1})$ and the corresponding map
\begin{align}
\hat T_\e(x) = \int yh''\left(\frac{\dd\tilde\pi_{\e}}{\dd(\lambda_{\mathsf s}\otimes \mu_{\mathsf s})}(x,y)\right)^{-1}\dd\mu_{\mathsf s}(y).
\end{align}
Note that $ \hat T_\e(x)$ enjoys the same regularity and regularity estimates as $ T_\e$ does (up to a constant depending on $|A^{-1}|$, $|A|$ and $|b|$ only). Moreover, $\lambda_{\mathsf s}(0)=\mu_{\mathsf s}(0)=1$. Due to the discussion in \cref{sec:affinechanges}, $\tilde \pi_\e$ is a minimiser of \eqref{problem} with a regularisation parameter bounded by $C\e$ for some $C>0$ between marginals $\lambda_{\mathsf s}, \mu_{\mathsf s}$ and it holds
\begin{align}
E(\tilde \pi_\e,R_0)+D_{\lambda_{\mathsf s},\mu_{\mathsf s}}(R_0)\lesssim \e_1.
\end{align}
Moreover, applying \cref{lem:eckstein}, if $p>1$,
\begin{align}
\int h\left(\frac{\dd\pi_{\mathsf s}}{\dd(\lambda_{\mathsf s}\otimes \mu_{\mathsf s}}\right)\dd(\lambda_{\mathsf s}\otimes \mu_{\mathsf s})\leq C \e^\frac 4 {d(p-1)+2}.
\end{align}
Reducing $\e_1$ further if necessary, this shows that
\begin{align}
\hat E(\tilde \pi_\e,R_0)+D_{\lambda_{\mathsf s},\mu_{\mathsf s}}(R_0)\lesssim \e_1.
\end{align}

In particular, we may apply \cref{thm:largescale} to $\tilde \pi_\e$ and deduce that choosing $\Lambda>0$ sufficiently large, there exist $\tilde A\in \R^{d\times d}$ and $\tilde b\in \R^d$ with 
\begin{align}
|\tilde b|+|\tilde A|\lesssim \e_1
\end{align}
such that with $r= \Lambda \e^\frac 2 {d(p-1)+2}$,
\begin{align}\label{eq:energySmallScale}
\frac 1 {r^{d+2}} \int_{\#_{r}}|y-x-\tilde Ax-\tilde b|^2 \dd\tilde \pi_\e \lesssim 1.
\end{align}

Consider the change of coordinates associated with $\mathsf s=(\tilde A+\textup{id},\tilde b,1,1)$. \cref{sec:affinechanges} shows that then $\pi_{\e,\mathsf s}$ is a minimiser of \eqref{problem} with a regularisation parameter bounded by $C\e$ for some $C>0$ and with marginals $\lambda_{\mathsf s}$ and $\mu_{\mathsf s}$, $\lambda_{\mathsf s}$ that are bounded and bounded away from $0$ on their support. Moreover, regularity and regularity estimates hold for
\begin{align}
T_\e^{\mathsf s} = \int yh''\left(\frac{\dd\pi_{\e,\mathsf s}}{\dd(\lambda_{\mathsf s}\otimes \mu_{\mathsf s})}(x,y)\right)^{-1}\dd\mu_{\mathsf s}(y)
\end{align}
if and only if they hold for $ \tilde T_\e$ (up to a constant depending on $|(\textup{id}+\tilde A)^{-1}|$, $|\tilde A+\textup{id}|$ and $|\tilde b|$ only). Due to \eqref{eq:energySmallScale}, we may apply \cref{prop:smallScale} to $T_\e^{\mathsf s}$, giving the desired Lipschitz estimate in $B_{Cr}$ for some universal $C$ small enough. We can then complete the proof in the case $\e \leq \e_1$ using a covering argument. The case $\e>\e_1$ is covered by \cref{lem:RegularitylargeEps}.
\end{proof}

We conclude this section by proving \cref{cor:convergence}.
\begin{proof}[Proof of \cref{cor:convergence}]
Let $(f_\e,g_\e)$ be the dual potentials for \eqref{problem}. \cref{lem:LipschitzRegularityPotential}  shows that $f_\e,g_\e$ are $L$-Lipschitz on $\supp\;\lambda \times \supp\;\mu$ where $L$ depends only on the size of the support. In order to see this, note that $(f_\e/\e^2,g_\e/\e^2)$ solve the dual problem with cost function $|x-y|^2/\e^2$, which gives a Lipschitz constant of order $1/\e^2$ for $(f_\e/\e^2,g_\e/\e^2)$ in \cref{lem:LipschitzRegularityPotential}. 

Let $x_0\in \supp\;\lambda$, $y_0\in \supp\;\mu$ and normalise by setting $f_\e(x_0)=0$. Then $\{f_\e\},\{g_\e\}$ are uniformly bounded in $\e$. By Arzela--Ascoli, we may extract a non-relabeled subsequence such that $\{f_\e\},\{g_\e\}$ converge uniformly to some limit $(f_\ast,g_\ast)$. 

Let now $(f,g)$ be the Kantorovich potential for the solution of $OT(\lambda,\mu)$. Recall that they solve the problem
\begin{align}\label{eq:dualProblemOT}
\mathcal D = \sup_{(f,g)\in C(\supp\;\lambda)\times C(\supp\;\mu),f(x)+g(y)\leq |x-y|^2} \int f \dd\lambda + \int g\dd\mu.
\end{align}

 Fixing $f(x_0)=0$, $(f,g)$ are unique. We further note that
 \begin{align}\label{eq:dualForm}
Z_\e = c_p\left(\frac{f_\e(x)+g_\e(y)-|x-y|^2}{\e^2}\right)_+^{\frac1 {p-1}}.
\end{align}
is the density with respect to $\lambda\otimes \mu$ of the optimal coupling $\pi_\e$ for \eqref{problem}. We now aim to show that $(f_\ast,g_\ast)$ is an admissible couple for the dual problem of optimal transport. Suppose $f_\ast(x)+g_\ast(y)>|x-y|^2$ for some $(x,y)\in \supp\;\lambda \times \supp\;\mu$. By continuity of $f_\ast,g_\ast,|x-y|^2$, there exists a compact neighbourhood $B$ of $(x,y)$ such that $f_\ast(x)+g_\ast(y)-|x-y|^2 >0$ in $B$. In light of \eqref{eq:dualForm}, $Z_\e\to+\infty$ uniformly on $B$ as $\e\to 0$. Noting $(\lambda\times \mu)(B)>0$, this contradicts $\int Z_\e \dd(\lambda\otimes \mu)=1$.

It remains to prove that $(f_\ast,g_\ast)$ are optimal. By uniqueness of the Kantorovich potentials, it will then follow that $(f_\ast,g_\ast)=(f,g)$. Note that by standard theory $OT(\lambda,\mu) = \mathcal D$. By
the duality theory developed in \cref{sec:regularised_optimal_transport} and using the notation established there, $\mathcal D_\e = OT_h(\lambda,\mu)$. Clearly $OT(\lambda,\mu)\leq OT_h(\lambda,\mu)$. Thus
\begin{align}
\mathcal D \leq& \mathcal D_\e = \int f_\e \dd\lambda + \int g_\e \dd \mu -c_p\e^2 \int \left(\frac{(f_\e(x)+g_\e(y)-|x-y|^2)_+}{\e^2}\right)^{p^\prime}\\
\leq& \int f_\e \dd\lambda + \int g_\e \dd \mu \to \int f_\ast \dd \lambda + \int g_\ast \dd \mu.
\end{align}
Thus $(f_\ast,g_\ast)$ solve \eqref{eq:dualProblemOT}.

Due to \cref{cor:potentials} (note the situation is symmetric in $(x,y)$ so an analogue statement holds for $g_\e$), $\{f_\e\},\{g_\e\}$ are uniformly bounded in $C^{2,1}_{\textup{loc}}(\supp\;\lambda)$ and $C^{2,1}_{\textup{loc}}(\supp\;\mu)$, respectively. As $(f_\e,g_\e)\to (f,g)$ uniformly on $\supp\;\lambda \times \supp\;\mu$, we deduce that $(f_\e,g_\e)\to (f,g)$ in $C^{2,1-}_{\textup{loc}}(\supp\;\lambda)$ and $C^{2,1-}_{\textup{loc}}(\supp\;\mu)$, respectively. Using \eqref{eq:BenamouBrenier} in the form of \cref{lem:concavity} and Brenier's theorem in the form $T(x) = x - \nabla f(x)$, it follows that $T_\e \to T$ in $C^{1,1-}_{\textup{loc}}(\supp\;\lambda)$.
\end{proof}

\appendix
\section{Regularised optimal transport} % (fold)
\label{sec:regularised_optimal_transport}
In this section, we prove some results about regularised optimal transport with general regularisations $h$ satisfying \cref{assumptionh}. We start with the following existence result. The arguments in this section closely follow those in \cite{Nutz2025} with the key difference being that we allow the regularisation to be a general convex function. We will always work under the assumption that $h$ satisfies \cref{assumptionh}. Under more restrictive assumptions on the regularisation the results in this section have been obtained in \cite{Bayraktar2025,Gonzalez2025}. For the sake of convenience we drop the dependence on $\e>0$.  We first note that~\eqref{problem} can be reformulated as follows:
\begin{align}
  \mathcal{P}\coloneqq &\,\min_{Z \in \mathcal{Z}}\mathsf{P}(Z)\\\coloneqq&\, \min_{Z \in \mathcal{Z}}\int_{\R^d \times \R^d} |x-y|^2 Z \, \dd P + \int_{\R^d \times \R^d} h(Z) \, \dd P \, ,
\end{align}
where 
\begin{equation}
  \mathcal{Z}\coloneqq \left\{ Z: Z \textrm{ is measurable, } Z\dd P \in \Pi(\lambda,\mu) \right\} \, .
\end{equation}
In what follows, we will use $c$ as a place holder for the quadratic cost $|x-y|^2$.
We refer to $\mathcal{P}$ as the \emph{primal problem}. Associated to $\mathcal{P}$, we have the following dual problem:
\begin{align}
  \mathcal{D}\coloneqq &\,\sup_{f\in L^1(\lambda), g \in L^1(\mu)}\mathsf{D}(f,g) \\\coloneqq&\,\sup_{f\in L^1(\lambda), g \in L^1(\mu)}\int_{\R^d}f \, \dd \lambda + \int_{\R^d}g \, \dd\mu - \int_{\R^d \times \R^d} h^\ast(f+g-c)\, \dd P \, .
\end{align}
We now have the following result.
\begin{proposition}\label{prop:primalExistence}
  Let $\lambda,\mu$ have compact support. Then, there exists a unique minimiser $Z^* \in L^1(\lambda \otimes \mu)$ of the primal problem.  
\end{proposition}
\begin{proof}
  Note that $\mathcal{Z}$ is closed, convex, and non-empty. Furthermore, the primal function is strongly lower semicontinous in $L^1(\lambda\otimes \mu)$ and thus, by convexity, weakly lower semicontinuous. Since $h$ is superlinear, sublevel sets of the primal functional are weakly compact in $L^1(\lambda\otimes \mu)$. Thus, by the direct method, there exists some minimiser in $\mathcal{Z}$. Furthermore, since $h$ is strictly convex the minimiser is unique, say $Z_*$. It is also clear that $\dd \pi=Z \,\dd(\lambda \otimes \mu)$ is the unique plan minimising~\eqref{problem}.
  \end{proof}

  We now move on to the the following weak duality result.
  \begin{lemma}\label{lem:weakDuality}
  Let $\lambda,\mu$ have compact support.    We have the following results:
    \begin{enumerate}
      \item For all $Z\in \mathcal{Z}$ and $(f,g)\in L^1(\lambda)\times L^1(\mu)$, we have that
      \begin{equation}
        \mathsf{P}(Z)\geq \mathsf{D}(f,g)\, ,
      \end{equation}
      with equality holding if and only if $Z=(h')^{-1}(f+g-c)$ $(\lambda\otimes\mu)$-a.e.
      \item Let $f,g$ be measurable, and $Z\in \mathcal{Z}$ be of the form $(h')^{-1}(f+g-c)$. Then,
      \begin{enumerate}
        \item $\mathsf{P}(Z)<\infty$ and $(f,g)\in L^1(\lambda)\times L^1(\mu)$.
        \item There is no duality gap, i.e. $\mathcal{P}=\mathcal{D}$.
        \item $Z$ and $(f,g)$ are optimal for $\mathsf{P}$ and $\mathsf{D}$, respectively. 
      \end{enumerate}
      \item Conversely, assume that $\mathcal{P}=\mathcal{D}$. Then, if $(f,g)\in L^1(\lambda) \times L^1(\mu)$ is optimal for $\mathcal{D}$ and $Z=(h')^{-1}(f+g-c)\in \mathcal{Z}$, then $Z$ is optimal for $\mathcal{P}$. Moreover
      \begin{align}\label{eq:schrodingerappendix}
      \int Z \,\dd\lambda(x) = 1 \quad \text{ for } \mu-\mathrm{a.e. }\, y,\quad
      \int Z \,\dd\mu(x)=1 \quad \text{ for } \lambda-\mathrm{a.e. }\,x.
      \end{align}
    \end{enumerate}
  \end{lemma}
\begin{proof}
  For (i), we note that $\mathsf{P}(1)<\infty$. Further,
  \begin{align}
    \mathsf{P}(Z)=&\, \int_{\R^d \times \R^d} (f+g)Z \, \dd(\lambda \otimes \mu)  -\int_{\R^d \times \R^d} (f+g-c)Z-h(Z) \, \dd(\lambda \otimes \mu) \, .
  \end{align}
  Note that the function $x\mapsto xy-h(x)$ has a unique maximum at $(h')^{-1}(y)=(h^\ast)'(y)$. It follows then that
  \begin{equation}
    \mathsf{P}(Z) \geq \mathsf{D}(f,g)\, ,
  \end{equation}
  with equality if and only if $Z=(h')^{-1}(f+g-c)$, $P$-a.e. Consequently, $\mathcal P\geq \mathcal D$. 
 Moreover, using the equality case in this calculation,
  \begin{align}
	\mathcal P\leq \mathsf P(Z) = \mathsf D(f,g)\leq \mathcal D.
  \end{align}  
  Hence (ii) is established. (iii) follows by the same argument. The moreover part is immediate from $Z\in \Pi(\lambda,\mu)$.
\end{proof}

In order to obtain a full duality result, it remains only to construct a solution $(f,g)$ to the dual problem satisfying $Z=(h')^{-1}(f+g-c)\in \mathcal Z$.
\begin{lemma}
There exist $(f,g)\in L^1(\lambda)\times L^1(\mu)$ solving the dual problem $\mathcal D$ such that $Z=(h')^{-1}(f+g-c)\in \mathcal Z$.
\end{lemma}
\begin{proof}
We consider the approximate problem: Fix $n\in \N$ and bounded measurable functions $\Phi =\{\phi_i\colon \R^d\times \R^d \to \R\}_{i=1,\ldots,n}$ with $\int_{\R^{d}\times \R^{d}} \phi_i \dd P = 0$. Consider
\begin{align}\label{eq:problemApproximate}
\inf_{Z\in \mathcal Z_n} \int c Z \,\dd P + \int h(Z) \,\dd P,
\end{align}
where 
\begin{align}
\mathcal Z_n = \{Z\colon Z \text{ is measurable}, \int_{\R^{d}\times \R^{d}} Z \dd P = 1, \int_{\R^{d}\times \R^{d}} \phi_i Z \dd P = 0, 1\leq i \leq n\}.
\end{align}
We claim that there exists a unique solution to \eqref{eq:problemApproximate} and that it takes the form $Z_n = (h')^{-1}(b_0+\sum_i b_i \phi_i-c)$ for some $b_i\in \R$, $0\leq i\leq n$.

Noting that $\mathcal Z_n$ is convex, closed and non-empty, existence and uniqueness follows analogously to the proof of \cref{prop:primalExistence}. Arguing as in the proof of \cref{lem:weakDuality}, we find
\begin{align}
\int_{\R^{d}\times \R^{d}} c Z \dd P + \int_{\R^{d}\times \R^{d}} h(Z)\dd P \geq b_0 +  \int_{\R^{d}\times \R^{d}} h^\ast(b_0+\langle b,\Phi\rangle-c)\dd P,
\end{align}
where $b_0\in \R$, $b=(b_1,\ldots, b_n)\in \R^n$ and $\Phi = (\phi_1,\ldots, \phi_n)$ with equality if and only if $Z=(h')^{-1}(b_0+b\cdot \Phi-c)$ for some $b_0\in \R$ and $b\in \R^n$. Consequently, it suffices to show that there exists $Z\in \mathcal Z_n$ of the form $Z=(h')^{-1}(b_0+b\cdot \Phi-c)$.

To this end, we first consider the problem
\begin{align}
\inf_{(b_0,b)\in \R^{n+1}} G(b_0,b),\quad G(b_0,b)\coloneqq -b_0 + \int h^\ast(b_0+\langle b,\Phi\rangle-c)\dd P.
\end{align}
We aim to show the existence of a minimiser $\mathbf b^\ast = (b_0^\ast,b^\ast)$.

As $h$ is continuous and real-valued, $h^\ast$ is continuous. Consequently, $G\colon \R^{n+1}\to \R$ is convex and continuous. Projecting onto the orthogonal complement of $\{b\in \R^n\colon \langle b,\Phi\rangle = 0,\; P-\text{a.s.}\}$, we may reduce to assuming that $\langle b,\Phi\rangle =0$, $P$-almost surely only for $b=0$. We claim that then $P(\{\langle b,\Phi\rangle>0\})>0$ for $b\neq 0$. If this is not the case for some $b\neq 0$, it must hold that $P(\{\langle b,\Phi\rangle <0\})>0$. In particular, then
\begin{align}
0=\int \langle b,\Phi\rangle \dd P= \int_{\langle b,\Phi\rangle <0} \langle b,\Phi\rangle \dd P<0.
\end{align}
 
As $h$ is bounded below, $\lim_{t\to+\infty} h^\ast(t)=+\infty$. Using that $P(\{\langle b,\Phi\rangle\}>0)>0$ for any $b\neq 0$, we deduce for any $\mathbf b\neq 0$,
\begin{align}
\lim_{\alpha\to+\infty} G(\alpha \mathbf b)=+\infty.
\end{align}
Consequently, applying \cite[Lemma 3.5]{Follmer1988}, we obtain the existence of a minimiser $\mathbf b^\ast$.

Let $\mathbf b^\ast=(b_0^\ast,b^\ast)$ be a minimiser of $G$. As $h$ is strictly convex and $C^1$ on $(0,+\infty)$, $h^\ast$ is $C^1$ on the range of $(h^\ast)'=( h')^{-1}$. Recalling that $\Phi$ is bounded, differentiating with respect to $b_i$ yields
\begin{align}\label{eq:conditionPhi}
\int (h')^{-1}(b_0^\ast+\langle b^\ast,\Phi\rangle -c)\dd P =1,\; \int \phi_i (h')^{-1}(b_0^\ast+\langle b^\ast,\Phi\rangle-c)\dd P = 0,\; 1\leq i\leq n.
\end{align}
In particular $(h')^{-1}(b_0^\ast+\langle b^\ast,\Phi\rangle-c)\in \mathcal Z_n$.

We now aim to find measurable functions $f_n,g_n\colon \R^d\to \R$ such that
\begin{gather}
 (h')^{-1}(f_n(x)+g_n(y)-c) \stackrel{n \to \infty}{\rightharpoonup} Z_\ast \quad \text{ weakly in } L^1(P)
\end{gather}
for some $Z_\ast$.
Applying \cite[Proposition 5.8]{Nutz} then concludes the proof, once we note that \eqref{eq:conditionPhi} combined with the above convergence, proves that $Z_\ast \dd P\in \Pi(\lambda,\mu)$.

We begin by noting that there are countably many, bounded measurable functions $\phi_k$ such that $\rho\in \mathscr P(\R^d\times \R^d)$ has first marginal $\lambda$ if and only if it holds that $\int \phi_k(x) \dd \rho(x,y)=0$ for all $k\geq 1$. Consider now $\Phi_n = \{\phi_i\colon 1\leq i \leq n\}=\{\phi_k(x),\phi_k(y)\colon k\geq 1\}$. Then $\rho\in \Pi(\lambda,\mu)$ if and only if $\int \phi_i \dd \rho = 0$ for all $i\geq 1$. Consider now the minimiser $Z_n$ of \eqref{eq:problemApproximate} in $\mathcal Z_n$ with $\Phi=\Phi_n$. Note that $Z_n$ is of the desired form. As $h$ is super-linear at infinity, it suffices to establish that $\int h(Z_n)\dd P \to \int h(Z)\dd P$.

Note that $1\in \mathcal Z_{n+1}\subset \mathcal Z_n$ for all $n\geq 1$. In particular, by construction, $\int (c Z_n+h(Z_n)) \dd P$ is a non-decreasing, bounded sequence and
\begin{align}
\lim_{n\to\infty} \int (c Z_n + h(Z_n)) \dd P\leq \mathcal D<\infty.
\end{align}
As $h$ is super-linear at infinity, we deduce that $\{Z_n\}$ is equi-integrable. Consequently, by Dunford--Pettis, up to picking a non-relabeled subsequence, $Z_n \rightharpoonup Z_\ast$ weakly in $L^1(P)$ for some $Z_\ast$.  This concludes the proof.
\end{proof}

In order to establish uniqueness of the potentials up to constant, we note the following continuity result. This additionally allows to find representatives such that \eqref{eq:schrodingerappendix} hold for every $(x,y)\in \supp\;\lambda\cap \supp\;\mu$.
\begin{lemma}\label{lem:LipschitzRegularityPotential}
Let $\lambda,\mu$ have compact support. Then the optimal potentials $(f,g)\in L^1(\lambda)\times L^1(\mu)$ have a Lipschitz extension to $\R^d$, with a Lipschitz constant $L$ that only depends on the size of the support of $\lambda$ and $\mu$. The extension to $\supp\;\lambda\times \supp\;\mu$ is unique.
\end{lemma}
\begin{proof}
For all $x,x'\in X_0\subset \supp\;\lambda$, where $X_0$ has full $\lambda$-measure, we find with $\Delta=\Delta(x,x')\coloneqq \sup_{y\in \supp\;\mu} ||x-y|^2-|x'-y|^2|$, using the monotonicity of $(h')^{-1}$,
\begin{align}
&\int (h')^{-1}(f(x)+g(y)-|x'-y|^2+ \Delta) \dd\mu(y) \\
\geq& \int (h')^{-1}(f(x)+g(y)-|x-y|^2)\dd\mu(y) = 1.
\end{align}
Similarly
\begin{align}
&\int (h')^{-1}(f(x)+g(y)-|x'-y|^2- \Delta) \dd\mu(y) \\
\leq& \int (h')^{-1}(f(x)+g(y)-|x-y|^2)\dd\mu(y) = 1.
\end{align}
Noting 
\begin{align}
1 = \int(h')^{-1}(f(x')+g(y)-|x'-y|^2)\dd\mu(y),
\end{align}
we deduce that $f(x')\in [f(x)-\Delta(x,x'),f(x)+\Delta(x,x')]$ and consequently,
\begin{align}
|f(x)-f(x')|\leq | |x-y|^2-|x'-y|^2| \lesssim |x-x'|,
\end{align}
where the implicit constant depends only on the size of the support of $\lambda$ and $\mu$.
By McShane's theorem, $f$ has a Lipschitz extension to $\R^d$. Note moreover that the Lipschitz extension to $\overline{X_0}\supset \supp\;\lambda$ is unique.
The argument for $g$ is analogous.
\end{proof}
We now proceed to prove uniqueness of the potentials.

\begin{lemma}
Assume there is $D>0$ such that
\begin{align}\label{eq:measureLower}
\sup_{B_r(z)\subset \supp\;\lambda} \frac{\lambda(B_r(z))}{|B_r(z)|}\geq D.
\end{align}
and assume $\textup{int}\;\supp\;\lambda$ is connected.
Then the solution $(f,g)\in L^1(\lambda)\times L^1(\mu)$ is unique up to constant.
\end{lemma}
\begin{proof}
Apply \cref{lem:LipschitzRegularityPotential} to obtain Lipschitz representatives of $(f,g)$ on $\R^d$, still denoted $(f,g)$. The representatives are uniquely determined from $(f,g)$ on $\supp\;\lambda \times \supp\;\mu$. Set $Z=(h')^{-1}(f(x)+g(y)-|x-y|^2)$. This is a version of the density of the optimal coupling $\pi_\ast$. Note that by the strict monotonicity of $h'$,
\begin{align}
Z>0 \quad \Leftrightarrow\;  h'(0)< f(x)+g(y)-|x-y|^2.
\end{align}
As $f,g$ are Lipschitz, we deduce that $E=\{Z>0\}$ is open and $\pi_\ast(E)=1$. Consider the projection onto the $x$-component $\Pi_x(E)$. Note that
\begin{align}
\lambda(\Pi_x(E)) = \pi_\ast(E)=1,\quad \supp\;\lambda =\overline{\Pi_x(E)}.
\end{align}
We deduce that since $\pi_\ast(E)$ is open, $\lambda(\textup{int}\;\supp\;\lambda)=1$.
Using \eqref{eq:measureLower}, we deduce that $\Pi_x(E)$ is a set of full  Lebesgue measure. By Rademacher's theorem, $f$ is differentiable on $\textup{dom}\;\nabla f$, a set of full Lebesgue measure. We set $\Lambda = \textup{dom}\;\nabla f \cap \Pi_x(E)\subset \supp\;\lambda$, a set of full Lebesgue measure on $\supp\;\lambda$.

We claim that $\nabla f$ is uniquely determined on $\Lambda$. Fix $x_0\in \Lambda$. Then $(x_0,y_0)\in E$ for some $y_0$. We note
\begin{align}
0<Z(x,y) = (h')^{-1}(f(x)+g(y)-|x-y|)\quad \forall (x,y)\in B_r(x_0,y_0)
\end{align}
for some $r>0$ as $E$ is open. Differentiating and noting $(h')^{-1}=(h^\ast)'$, we obtain
\begin{align}
\nabla_x Z(x_0,y_0) = (h^\ast)''(f(x_0)+g(y_0)-|x_0-y_0|)(\nabla f(x_0)-(x_0-y_0)).
\end{align}
Note $(h^\ast)''(f(x_0)+g(y_0))-|x_0-y_0|) = h''(Z(x_0,y_0))>0$, so that
\begin{align}
\nabla f(x_0) = \frac{\nabla_x Z(x_0,y_0)}{h''(Z(x_0,y_0))}+(x_0-y_0).
\end{align}
The right-hand side is uniquely determined as $\pi_\ast$ is unique. Consequently, $f$ is a Lipschitz function, whose gradient is uniquely determined almost everywhere on the  open and connected set $\text{int}\;\supp\;\lambda$. This implies that $f$ is uniquely determined up to constant on $\textup{int}\;\supp\;\lambda$, which completes the proof as this set has full $\lambda$-measure.
\end{proof}

% section regularised_optimal_transport (end)
\bibliographystyle{amsalpha}

\end{document}